\documentclass[11pt]{article}

\usepackage{colonequals}

\usepackage{lscape}
\usepackage{pdflscape}

\usepackage{stmaryrd}

\usepackage{cmll}

\usepackage{comment}
\usepackage{url}

\usepackage{amscd}
\usepackage{amssymb}
\usepackage{amsmath}
\usepackage{amsthm}
\usepackage{mathrsfs}

\usepackage{color}
\usepackage{xspace}
\usepackage{bussproofs}
\EnableBpAbbreviations

\usepackage{mathtools}

\usepackage{bpextra}

\usepackage{enumerate}

\usepackage{centernot}

\usepackage{hyperref}

\usepackage{cleveref}

\usepackage{caption}

\usepackage{multirow}

\usepackage{array}
\newcolumntype{C}[1]{>{\centering\arraybackslash}p{#1}}
\newcolumntype{L}[1]{>{\arraybackslash}p{#1}}

\usepackage{color}
\usepackage{cancel}
\usepackage{verbatim,cmll}
\usepackage{setspace}
\usepackage{lscape}
\usepackage{latexsym,relsize}
\usepackage{multicol}

\usepackage[position=top]{subfig}
\usepackage{leqno}
\usepackage{tikz}

\usepackage{txfonts}

\usepackage{authblk}
\usepackage{multicol}
\usetikzlibrary{arrows}
\usetikzlibrary{matrix}
\usetikzlibrary{patterns}
\usetikzlibrary{shapes}

\newcommand{\marginnote}[1]{\marginpar{\raggedright\tiny{#1}}}
\def\fCenter{{\mbox{$\ \Rightarrow\ $}}}

\newcommand{\OFNEG}{\hat{{\sim}}}
\newcommand{\ofneg}{{\sim}}

\newcommand{\OGNEG}{\check{\neg}}
\newcommand{\ogneg}{\neg}

\newcommand{\ONEG}{\tilde{\neg}}

\newcommand{\fns}{\footnotesize}

\newcommand{\mand}{\otimes}

\newcommand{\mrarr}{\,\,\backslash\,\,}

\newcommand{\mlarr}{\,\,\slash\,\,}

\newcommand{\MAND}{\,\hat{\otimes}\,}
\newcommand{\MRARR}{\,\,\check{\backslash}\,\,}
\newcommand{\MLARR}{\,\check{\slash}\,}

\newcommand{\aatop}{\top}
\newcommand{\abot}{\bot}
\newcommand{\aand}{\wedge}
\newcommand{\aor}{\vee}

\newcommand{\fneg}{\ensuremath{\sim}\xspace}
\newcommand{\gneg}{\ensuremath{\neg}\xspace}

\newcommand{\AATOP}{\hat{\top}}
\newcommand{\ABOT}{\ensuremath{\check{\bot}}\xspace}

\newcommand{\wbox}{\ensuremath{\Box}\xspace}
\newcommand{\wdia}{\ensuremath{\Diamond}\xspace}

\newcommand{\bbox}{\ensuremath{\blacksquare}\xspace}
\newcommand{\bdia}{\ensuremath{\Diamondblack}\xspace}
\newcommand{\WBOX}{\ensuremath{\check{\Box}\:}\xspace}
\newcommand{\WDIA}{\ensuremath{\hat{\Diamond}}\xspace}
\newcommand{\BBOX}{\ensuremath{\check{\blacksquare}}\xspace}
\newcommand{\BDIA}{\ensuremath{\hat{\Diamondblack}}\xspace}

\newcommand{\FH}{\ensuremath{\hat{f}}\xspace}

\newcommand{\GH}{\ensuremath{\hat{g}}\xspace}
\newcommand{\GC}{\check{g}\xspace}
\newcommand{\FHS}{\ensuremath{\hat{f}^{\,\sharp}}\xspace}
\newcommand{\FCS}{\ensuremath{\check{f}^{\,\sharp}}\xspace}
\newcommand{\GHF}{\ensuremath{\hat{g}^{\,\flat}}\xspace}
\newcommand{\GCF}{\ensuremath{\check{g}^{\,\flat}}\xspace}
\newcommand{\fs}{\ensuremath{f^\sharp}\xspace}

\newcommand{\starfor}{{/\!\!}_{\star}}
\newcommand{\circfor}{{/\!}_{\circ}}

\newcommand{\starback}{\backslash_{\star}}
\newcommand{\circback}{\backslash_{\circ}}

\newcommand{\F}{\mathsf{F}}
\newcommand{\G}{\mathsf{G}}

\renewcommand{\epsilon}{\varepsilon}

\newcommand{\vp}{\overline{p}}

\newcommand{\ophi}{\overline{\varphi}}
\newcommand{\opsi}{\overline{\psi}}

\newcommand{\OGA}{\overline{\Gamma}}
\newcommand{\ODE}{\overline{\Delta}}
\newcommand{\OSI}{\overline{\Sigma}}

\theoremstyle{plain}
\newtheorem{thm}{Theorem}

\newtheorem{lem}[thm]{Lemma}
\newtheorem{cor}[thm]{Corollary}
\newtheorem{prop}[thm]{Proposition}

\newtheorem{lemma}[thm]{Lemma}
\newtheorem{example}[thm]{Example}

\theoremstyle{definition}
\newtheorem{remark}[thm]{Remark}
\newtheorem{definition}[thm]{Definition}

\newcommand{\bbL}{\mathbb{L}}
\newcommand{\bba}{\mathbb{A}}
\newcommand{\bbA}{\mathbb{A}}
\newcommand{\bari}[1]{\overline{#1}^{\, i}}

\newcommand{\blhd}{\blacktriangleleft}
\newcommand{\brhd}{\blacktriangleright}

\newcommand{\cceq}{\coloncolonequals}
\newcommand{\ceq}{\colonequals}

\makeindex
\title{Algebraic proof theory for $\mathrm{LE}$-logics\thanks{This project has received funding from the European Union's Horizon 2020 research and innovation programme under the Marie Sk\l{}odowska-Curie grant agreement No 101007627.}}
\author[2]{Giuseppe Greco}
\author[3]{Peter Jipsen}
\author[1,5]{Fei Liang\thanks{The research of the third author is supported by the Chinese Ministry of Education of Humanities and Social Science Project (23YJC72040003) and the Young Scholars Program of Shandong University (11090089964225).}}
\author[2,4]{Alessandra Palmigiano\thanks{The research of the first and fourth author is partially funded by the NWO grant KIVI.2019.001.}}
\author[2,5]{Apostolos Tzimoulis\thanks{The research of the third and fifth authors is supported by the Key Project of Chinese Ministry of Education (22JJD720021).}}

\affil[1]{\small School of Philosophy and Social Development, Shandong University, China}
\affil[2]{\small Vrije Universiteit Amsterdam, The Netherlands}
\affil[3]{\small Chapman University, USA}
\affil[4]{\small Department of Mathematics and Applied Mathematics, University of Johannesburg, South Africa}
\affil[5]{\small Institute of Logic and Cognition, Sun Yat-Sen University, China}

\date{}

\begin{document}
\maketitle

\begin{abstract}
In this paper we extend the research programme in algebraic proof theory from axiomatic extensions of the full Lambek calculus to logics algebraically captured by certain varieties of normal lattice expansions (normal LE-logics). Specifically, we generalise the {\em residuated frames} in  \cite{galatos2013residuated} to arbitrary signatures of normal lattice expansions ($\mathrm{LE}$).  Such a generalization provides a valuable tool for proving important properties of $\mathrm{LE}$-logics in full uniformity. We prove semantic cut elimination for the display calculi $\mathrm{D.LE}$ associated with the basic normal LE-logics and their axiomatic extensions with analytic inductive axioms. We also prove the finite model property (FMP) for each such calculus $\mathrm{D.LE}$, as well as for its extensions with analytic structural rules satisfying certain additional properties.  

\ 

{\fns \noindent \textbf{Keywords:} algebraic proof theory; polarity based semantics; normal lattice expansions; non-distributive logics; cut-elimination; finite model property; display sequent calculi; substructural logics.}
\end{abstract}


\section{Introduction}
Algebraic proof theory\cite{ciabattoni2012algebraic} is a research area aimed at establishing systematic connections between results and insights in structural proof theory (such as  cut elimination theorems) and in algebraic logic (such as representation theorems for classes of algebras). While results of each type have been traditionally formulated and developed independently from the other type, algebraic proof theory aims to integrate these fields. The main results in algebraic proof theory mainly concern axiomatic extensions of the full Lambek calculus, and, building on the work of many authors \cite{belardinelli2004algebraic,ciabattoni2006towards,terui2007structural,galatos2010cut,ciabattoni2012algebraic,galatos2013residuated}, establish a systematic connection between a strong form of cut elimination for certain substructural logics (on the proof-theoretic side) and the closure of their corresponding varieties of algebras under MacNeille completions  (on the algebraic side). Specifically, given a cut eliminable sequent calculus  for a basic logic (e.g.~the full Lambek calculus), a core question in structural proof theory concerns the identification of axioms which can be added to the given basic logic so that the resulting axiomatic extensions can be captured by calculi which are again cut eliminable (in what follows, such axiomatic extensions will be referred to as {\em analytic} axiomatic extensions). This question is difficult, since the cut elimination theorem is notoriously a very fragile result. However, in \cite{ciabattoni2008axioms,ciabattoni2012algebraic} a very satisfactory answer is given to this question for substructural logics, by identifying  a hierarchy $(\mathcal{N}_n, \mathcal{P}_n)$ of axioms in the language of the full Lambek calculus, referred to as the {\em substructural hierarchy}, and guaranteeing that, up to the level $\mathcal{N}_2$, these axioms can be effectively transformed into special structural rules (called {\em analytic}) which can be safely added to a cut eliminable calculus without destroying cut elimination.
Algebraically, this transformation corresponds to the possibility of transforming equations into equivalent quasiequations, and remarkably, such a transformation (which we will expand on shortly) is also  key to proving preservation under MacNeille completions and canonical extensions \cite{JonTar51, gehrke2001bounded}.
The second major contribution of algebraic proof theory is the identification of the semantic (algebraic) essence of cut elimination (for cut-free sequent calculi for substructural logics) in the relationship between  certain polarity-based relational structures (referred to as {\em residuated frames}) $\mathbb{W}$ arising from the given sequent calculus, and  certain ordered algebras $\mathbb{W}^+$ which can be thought of as the complex algebras of $\mathbb{W}$ by analogy with modal logic.
Specifically, the fact that the calculus is cut-free is captured semantically by  $\mathbb{W}$ being an {\em intransitive} structure, while $\mathbb{W}^+$ is by construction an ordered algebra, on which the cut rule is sound. Hence, in this context, cut elimination is encoded in the preservation of validity from $\mathbb{W}$ to $\mathbb{W}^+$. For instance,  the validity of analytic structural rules/quasiequations is preserved from $\mathbb{W}$ to $\mathbb{W}^+$ (cf.~\cite{ciabattoni2012algebraic}), which shows that analytic structural rules can indeed be safely added to the basic Lambek calculus in a way which preserves its cut elimination. 


In \cite{galatos2013residuated}, residuated frames are introduced. Much in the same way as Kripke frames for modal logic, residuated frames provide relational semantics for substructural logics and underlie the representation theory for the algebraic semantics of substructural logics. The algebraic proof theory program is developed in \cite{galatos2013residuated} by showing the existence of a connection
between Gentzen-style sequent calculi for substructural logics and residuated frames, which translates into a connection between a cut-free proof system, the finite model property and the finite
embeddability property for the corresponding variety of algebras.

A closely related but different line of investigation motivated by the same general question (concerning the identification of classes of analytic axiomatic extensions of given basic logics) has been recently pursued in the setting of proper display calculi \cite{Wan02} for {\em normal (D)LE-logics}, i.e.~the logics algebraically captured by varieties of normal (distributive) lattice expansions\footnote{This class of logics prominently includes (bi-)intuitionistic logic, modal logics on a classical and non-classical (e.g.~intuitionistic, distributive, general lattice) propositional base, substructural logics, quantum logic, paraconsistent logics such as De Morgan and semi De Morgan logics, etc.} (LEs), and especially in connection with the semantic theory of  {\em generalized Sahlqvist theory for (D)LE-logics} \cite{conradie2012algorithmic, CoGhPa14, conradie2016constructive, conradie2019algorithmic}. Originating in an observation of Kracht's in the setting of proper display calculi for classical normal modal logics \cite{Kracht}, this line of investigation was further developed in \cite{greco2018unified, ChnGrePalTzi21}. In \cite{greco2018unified},  the same algorithm for computing the first order correspondents of axioms in an arbitrary (D)LE-signature  was also used to generate their equivalent analytic structural rules (which preserve the applicability of Belnap's general strategy for {\em syntactic} cut elimination when added to a proper display calculus). In the same paper, properly displayable (D)LE-logics (i.e.~those logics which can be captured by a proper display calculus) were characterized in terms of the syntactic shape of a proper subclass of (generalized) Sahlqvist axioms, namely the {\em analytic inductive LE-axioms} (cf.~\cite[Definition 55]{greco2018unified}). Also, thanks to the connection with generalized Sahlqvist theory, a set of basic properties besides subformula property and  Belnap-style cut elimination (namely soundness, completeness, and conservativity) was shown to uniformly hold for the proper display calculi associated with the logics of this class. In particular, the proof of conservativity  hinges on the fact that the validity of generalized Sahlqvist LE-axioms (and hence also of analytic inductive LE-axioms, which form a proper subclass thereof) is preserved under the {\em canonical extension} construction (cf.~\cite[Theorem 7.1]{conradie2019algorithmic}),
and moreover, that canonical extensions of normal LEs are {\em fully residuated} algebras, i.e.~the adjoints and residuals of each connectives in each coordinate exist, even if they might not exist in the original algebra.\footnote{This fact is the algebraic generalization of the well known fact that classical tense modal logic is conservative over classical normal modal logic, since Kripke frames are also frames for tense logic.}

In this paper, we bring together the two lines of investigation discussed above: we extend results and
techniques  in algebraic proof theory from substructural logics to
 normal LE-logics, also using results and insights from generalized Sahlqvist theory. The broadness of this setting  makes it possible for techniques and results 
to transfer from one area to another; for instance, correspondence-theoretic results developed for modal logics can be transferred to substructural logics, and conversely, proof-theoretic results developed for substructural logics can be transferred to modal logics.  Concretely:
\begin{enumerate}
\item  building on the {\em polarity-based} semantics for LE-logics \cite{conradie2016categories, CFPPTW17, conradie20}, we introduce \emph{LE-frames} as the counterparts of residuated frames of \cite{galatos2013residuated}  for arbitrary normal lattice expansion signatures (LE-signatures) which do not need to be closed under the residuals of each
connective;
\item we introduce \emph{functional D-frames} as the LE-frames associated with any proper display calculus in any  LE-signature; this generalization involves moving from structural rules of so-called simple shape to the more general class of analytic
structural rules (cf.\ \cite{greco2018unified}, Definition 4) in any LE-signature.
\end{enumerate}
The contributions of the present paper include:
\begin{enumerate}
\item the proof of semantic cut elimination for the display calculus $\mathrm{D.LE}$ associated with the basic normal
LE-logic 
in any normal LE-signature; 
\item  the transfer of the cut elimination result to extensions of $\mathrm{D.LE}$
with analytic structural rules; 
\item the finite model property for $\mathrm{D.LE}$ and for  extensions of $\mathrm{D.LE}$
with analytic structural rules satisfying certain additional properties. 
\end{enumerate}
We also discuss how these results recapture the semantic cut elimination results in \cite{ciabattoni2012algebraic} and apply in a modular way to a range of logics which includes the basic epistemic logic of categories and its analytic extensions, the full Lambek calculus and its analytic extensions, the Lambek-Grishin calculus and its analytic extensions, and orthologic.

The paper is organized as follows. In Section \ref{subset:language:algsemantics}, we gather preliminary notions on $\mathrm{LE}$-logics, their syntax, algebraic semantics, and display calculi. In Section \ref{sec:leframes}, we introduce relational models for $\mathrm{LE}$-logics, $\mathrm{LE}$-frames. In Section \ref{sec:dframes} we introduce $\mathrm{D}$-frames, to semantically interpret cut-free display calculi. In Section \ref{sec:semcut}, we prove semantically cut elimination for the display calculi of $\mathrm{LE}$-logics and their analytic extensions. In Section \ref{sec:fmp} we provide a general result for finite model property for certain classes of $\mathrm{LE}$-logics and in Section \ref{sec:exfmp} we provide a number of examples that fit in this class. Finally, in Section \ref{sec:concl} we summarize the results of this article
and collect further research directions.

\section{Preliminaries}
\label{subset:language:algsemantics}
In this section we recall definition, notation and basic properties of $\mathrm{LE}$-logics. As discussed in the introduction, this setting uniformly accounts for many well known logical systems. 
This section reports on and adapts material from \cite{conradie2019algorithmic,greco2018unified}. We start by introducing the language of $\mathrm{LE}$-logics, their algebraic interpretation on normal lattice expansions, and a complete and sound sequent-based axiomatization, and their expansions to a fully residuated language. We continue by introducing polarity-based relational semantics for $\mathrm{LE}$-logics, $\mathrm{LE}$-frames and explain how they can interpret via their algebraic duals $\mathrm{LE}$-logics.  Finally, we introduce the display calculi for $\mathrm{LE}$-logics and their extensions with analytic structural rules. Lattices with residuated operations are closely connected with \emph{partial gaggles} developed in \cite{Dun90}, while their connection with display calculi was originally investigated in \cite{Gore98Gaggles}. In Appendix \ref{sec:GagglesTheoryBasicNotionsAndNomenclature} we provide a thorough comparison between gaggle theory and the theory of lattice expansions.

	\subsection{Basic normal $\mathrm{LE}$-logics and their algebras}\label{ssec:basic}
	
 \paragraph{Language and axiomatization of basic normal LE-logics.}
An {\em order-type} over $n\in \mathbb{N}$ is an $n$-tuple $\epsilon\in \{1, \partial\}^n$. For every order type $\epsilon$, we denote its {\em opposite} order type by $\epsilon^\partial$, that is, $\epsilon^\partial_i = 1$ iff $\epsilon_i=\partial$ for every $1 \leq i \leq n$. For any lattice $\bba$, we let $\bba^1: = \bba$ and $\bba^\partial$ be the dual lattice, that is, the lattice associated with the converse partial order of $\bba$. For any order type $\varepsilon$ over $n$, we let $\bba^\varepsilon: = \Pi_{i = 1}^n \bba^{\varepsilon_i}$.
	
The language $\mathcal{L}_\mathrm{LE}(\mathcal{F}, \mathcal{G})$ (from now on abbreviated as $\mathcal{L}_\mathrm{LE}$) takes as parameters: a denumerable set of proposition letters $\mathsf{AtProp}$, elements of which are denoted $p,q,r$, possibly with indexes, and disjoint sets of connectives $\mathcal{F}$ and $\mathcal{G}$.\footnote{\label{footnote:fg}
The connectives in $\mathcal{F}$ (resp.\ $\mathcal{G}$) correspond to those referred to as {\em positive} (resp.\ {\em negative}) connectives in \cite{ciabattoni2008axioms}. This terminology is not adopted in the present paper to avoid confusion with other usages of these adjectives throughout the paper.
Our assumption that the sets $\mathcal{F}$ and $\mathcal{G}$ are disjoint is motivated by the desideratum of generality and modularity. Indeed, for instance, the order theoretic properties of Boolean negation $\neg$ guarantee that this connective belongs both to $\mathcal{F}$ and to $\mathcal{G}$. In such cases we prefer to define two copies $\neg_\mathcal{F}\in\mathcal{F}$ and $\neg_\mathcal{G}\in\mathcal{G}$, and introduce structural rules (see Section \ref{ssec:syntactic frames associated w algebras}) which encode the fact that these two copies coincide.} Each $f\in \mathcal{F}$ (resp.~$g\in \mathcal{G}$) has arity $n_f\in \mathbb{N}$ (resp.\ $n_g\in \mathbb{N}$) and is associated with some order-type $\varepsilon_f$ over $n_f$ (resp.\ $\varepsilon_g$ over $n_g$). 
Unary connectives $f$ (resp.\ $g$)  are sometimes denoted as $\Diamond$ (resp.\ $\Box$) if their order-type is 1, and $\lhd$ (resp.\ $\rhd$) if their order-type is $\partial$.\footnote{The adjoints of the unary connectives $\Box$, $\Diamond$, $\lhd$ and $\rhd$ are sometimes denoted $\Diamondblack$, $\blacksquare$, $\blhd$ and $\brhd$, respectively.} The terms (formulas) of $\mathcal{L}_\mathrm{LE}$ are defined recursively as follows:
	\[
	\varphi \cceq p \mid \bot \mid \top \mid \varphi \wedge \varphi \mid \varphi \vee \varphi \mid f(\varphi_1, \ldots, \varphi_{n_f}) \mid g(\varphi_1, \ldots, \varphi_{n_g})
	\] 
	where $p \in \mathsf{AtProp}$, $f \in \mathcal{F}$, $g \in \mathcal{G}$, and $\top$ and $\abot$ are optional. Terms in $\mathcal{L}_\mathrm{LE}$ will be denoted either by $s,t$, or by lowercase Greek letters such as $\varphi, \psi, \gamma$ etc. 
In the remainder of the paper, when it is clear from the context,
we will often simplify notation and write e.g.\  $n$ for $n_f$ and $\varepsilon_i$ for $\varepsilon_{f,i}$. We also extend the $\{1,\partial\}$-notation to the symbols $\vee,\wedge,\bot,\top,\le,\vdash$ by defining 
$$
\vee^\partial=\wedge,\qquad \wedge^\partial=\vee,\qquad \bot^\partial=\top,\qquad \top^\partial=\bot,\qquad {\le^\partial}={\ge},\qquad {\vdash^\partial}={\dashv}
$$
while superscript $^1$ denotes the identity map. Therefore, in what follows, we will sometimes write e.g.~$\vee^{\epsilon_i}$ to denote $\vee$ when $\epsilon_i = 1$ and  $\wedge$ when $\epsilon_i = \partial$.

In what follows, for every $k \in \mathcal{F} \cup \mathcal{G}$ we use $k(\ophi)[\varphi]_i$ to indicate that the formula $\varphi$ occurs in the $i$-th coordinate of the vector $\ophi$.

  The generic LE-logic is not equivalent to a sentential logic. Hence the consequence relation of these logics cannot be uniformly captured in terms of theorems, but rather in terms of sequents, which motivates the following definition:
		For any language $\mathcal{L}_\mathrm{LE} = \mathcal{L}_\mathrm{LE}(\mathcal{F}, \mathcal{G})$, the {\em basic}, or {\em minimal} $\mathcal{L}_\mathrm{LE}$-{\em logic} is a set of sequents $\varphi\vdash\psi$, with $\varphi,\psi\in\mathcal{L}_\mathrm{LE}$, which contains as axioms the following sequents for lattice operations and additional connectives:
\begin{center}
\begin{tabular}{c}
$\bot\vdash p, \qquad\quad p\vdash p, \qquad\quad p\vdash \top,$ \\
\rule[0mm]{0mm}{6mm}$p\vdash p \vee q, \quad\quad q\vdash p\vee q, \quad\quad p\wedge q\vdash p, \quad p\wedge q \vdash q,$ \\
\rule[0mm]{0mm}{6mm}$f(\vp)[q\vee^{\epsilon_i} r]_i \vdash f(\vp)[q]_i \vee f(\vp)[r]_i, \qquad f(\vp)[\bot^{\epsilon_i}]_i \vdash \bot,$\\
\rule[0mm]{0mm}{6mm}$g(\vp)[q]_i \wedge g(\vp)[r]_i \vdash g(\vp)[q\wedge^{\epsilon_i} r]_i, \qquad \top \vdash g(\vp)[\top^{\epsilon_i}]_i,$\\
\end{tabular}
\end{center}
		and is closed under the following inference rules (note that $\varphi\vdash^\partial\psi$ means $\psi\vdash\varphi$):
\begin{center}
\begin{tabular}{@{}c@{}}
\AXC{$\varphi \vdash \chi$}
\AXC{$\chi \vdash \psi$}
\BIC{$\varphi \vdash \psi$}
\DP

\AXC{$\varphi \vdash \psi$}
\UIC{$\varphi(\chi/p) \vdash \psi(\chi/p)$}
\DP

\AXC{$\chi \vdash \varphi$}
\AXC{$\chi \vdash \psi$}
\BIC{$\chi \vdash \varphi \wedge \psi$}
\DP

\AXC{$\varphi \vdash \chi$}
\AXC{$\psi \vdash \chi$}
\BIC{$\varphi \vee \psi \vdash \chi$}
\DP
 \\
 
\rule[0mm]{0mm}{8mm}\AXC{$\varphi \vdash^{\epsilon_{f,i}} \psi$}
\UIC{$f(\vp)[\varphi]_i \vdash f(\vp)[\psi]_i$}
\DP

\qquad\qquad

\AXC{$\varphi \vdash^{\epsilon_{g,i}} \psi$}
\UIC{$g(\vp)[\varphi]_i \vdash g(\vp)[\psi]_i$}
\DP
 \\
\end{tabular}
\end{center}


In a basic $\mathcal{L}_{\mathrm{LE}}(\mathcal{F},\mathcal{G})$, the elements of $\mathcal{F}\cup\mathcal{G}$ are mutually independent. However, in some cases we might require that some pairs of  connectives in $\mathcal{F}\cup\mathcal{G}$ are one another's residuals in some coordinates. In particular, given $f\in\mathcal{F}$ or $g\in\mathcal{G}$ we might have $f^\sharp_i\in\mathcal{G}$ if $\varepsilon_{f,i} = 1$ or $g^\flat_i\in\mathcal{F}$ if $\varepsilon_{g,i} = 1$, and $f^\sharp_i\in\mathcal{F}$ if $\varepsilon_{f,i} = \partial$ or $g^\flat_i\in\mathcal{G}$ if $\varepsilon_{g,i} = \partial$, the order-type of which are as follows
	\begin{itemize}
		\item $\epsilon_{f_i^\sharp,i} = \epsilon_{f,i}$ and $\epsilon_{f_i^\sharp,j} = (\epsilon_{f,j})^{\epsilon_{f,i}^\partial}$ for any $j\neq i$,
		\item $\epsilon_{g_i^\flat,i} = \epsilon_{g,i}$ and $\epsilon_{g_i^\flat,j} = (\epsilon_{g,j})^{\epsilon_{g,i}^\partial}$ for any $j\neq i$.
	\end{itemize}
 The $n_f$-ary connective $f^\sharp_i$ is the intended interpretation of the right residual of $f\in\mathcal{F}$ in its $i$th coordinate if $\varepsilon_{f,i} = 1$ (resp.\ its Galois-adjoint if $\varepsilon_{f,i} = \partial$). The $n_g$-ary connective $g^\flat_i$ is the intended interpretation of the left residual of $g\in\mathcal{G}$ in its $i$th coordinate if $\varepsilon_{g,i} = 1$ (resp.\ its Galois-adjoint if $\varepsilon_{g,i} = \partial$). For instance, if $f$ and $g$ are binary connectives such that $\varepsilon_f = (1, \partial)$ and $\varepsilon_g = (\partial, 1)$, then $\varepsilon_{f^\sharp_1} = (1, 1)$, $\varepsilon_{f^\sharp_2} = (1, \partial)$, $\varepsilon_{g^\flat_1} = (\partial, 1)$ and $\varepsilon_{g^\flat_2} = (1, 1)$.\footnote{Note that this notation depends on the connective which is taken as primitive, and needs to be carefully adapted to well known cases. For instance, consider the  `fusion' connective $\circ$ (which, when denoted  as $f$, is such that $\varepsilon_f = (1, 1)$). Its residuals
$f_1^\sharp$ and $f_2^\sharp$ are commonly denoted $/$ and
$\backslash$ respectively. However, if $\backslash$ is taken as the primitive connective $g$, then $g_2^\flat$ is $\circ = f$, and
$g_1^\flat(x_1, x_2): = x_2/x_1 = f_1^\sharp (x_2, x_1)$. This example shows
that, when identifying $g_1^\flat$ and $f_1^\sharp$, the conventional order of the coordinates is not preserved, and depends on which connective
is taken as primitive.}

In this case the basic logic is augmented with the following rules
$$
			\begin{array}{cc}
			\AXC{$f(\ophi)[\varphi]_i \vdash \psi$}
			\doubleLine
			\UIC{$\varphi\vdash^{\epsilon_{f,i}} f^\sharp_i(\ophi)[\psi]_i$}
			\DP
			\qquad&\qquad
			\AxiomC{$\varphi \vdash g(\ophi)[\psi]_i$}
			\doubleLine
			\UIC{$g^\flat_i(\ophi)[\varphi]_i \vdash^{\epsilon_{g,i}} \psi$}
			\DP
			\end{array}
			$$
The double line in each rule above indicates that the rule is invertible (i.e., bidirectional).

Any given language $\mathcal{L}_\mathrm{LE} = \mathcal{L}_\mathrm{LE}(\mathcal{F}, \mathcal{G})$ can be associated with the language $\mathcal{L}_\mathrm{LE}^* = \mathcal{L}_\mathrm{LE}(\mathcal{F}^*, \mathcal{G}^*)$, where $\mathcal{F}^*\supseteq \mathcal{F}$ and $\mathcal{G}^*\supseteq \mathcal{G}$ are obtained by expanding $\mathcal{L}_\mathrm{LE}$ with residuals of each connective at each coordinate. Then, the logic $\mathbf{L}_{\mathrm{LE}}$ is expanded to $\mathbf{L}_\mathrm{LE}^*$, the minimal fully residuated $\mathcal{L}_\mathrm{LE}$-logic, by adding the corresponding residuation rules.
		\begin{thm}(\cite[Theorem 2.4]{ChnGrePalTzi21})
			\label{th:conservative extension}
			The logic $\mathbf{L}_\mathrm{LE}^*$ is a conservative extension of $\mathbf{L}_\mathrm{LE}$, i.e.~every $\mathcal{L}_\mathrm{LE}$-sequent $\varphi\vdash\psi$ is derivable in $\mathbf{L}_\mathrm{LE}$ if and only if $\varphi\vdash\psi$ is derivable in $\mathbf{L}_\mathrm{LE}^*$. 
		\end{thm}

\begin{example}\label{ex:1}
    As a running example we consider the language $\mathcal{L}_\mathrm{LE}(\{\mand\},\{\Box,\mrarr\})$, with $n_{\Box}=1$, $n_{\mand} = n_{\mrarr} = 2$, $\varepsilon_{\Box,1}=1$, $\varepsilon_{\mand, 1} = \varepsilon_{\mand, 2} = \varepsilon_{\mrarr, 2} = 1$, $\varepsilon_{\mrarr, 1} = \partial$  and where the logic contains the following bidirectional rule  

 $$ 		\begin{array}{c}
			\AXC{$\varphi\mand\psi\vdash \sigma$}
			\doubleLine
			\UIC{$\psi\vdash\varphi\mrarr \sigma$}
   \DP
   \end{array}$$
The fully residuated language is given by $\mathcal{F}^\ast=\{\bdia,\mand\}$ and $\mathcal{G}^\ast=\{\wbox,/,\backslash\}$ and the logic is augmented with the following rules
   $$
			\begin{array}{ccc}
			\AXC{$\varphi\mand\psi\vdash \sigma$}
			\doubleLine
			\UIC{$\varphi\vdash\sigma\mlarr\psi$}
			\DP
			\qquad&\qquad
			\AxiomC{$\varphi\vdash\sigma\mlarr\psi$}
			\doubleLine
			\UIC{$\psi\vdash\varphi\mrarr \sigma$}
			\DP
   \qquad&\qquad
   \AxiomC{$\bdia\varphi\vdash\psi$}
			\doubleLine
			\UIC{$\varphi\vdash\wbox\psi$}
			\DP
			\end{array}
			$$
\end{example}

\medskip

We let $\mathbf{L}_\mathrm{LE}(\mathcal{F}, \mathcal{G})$ denote the minimal $\mathcal{L}_\mathrm{LE}(\mathcal{F}, \mathcal{G})$-logic. We typically drop reference to the parameters when they are clear from the context. By an {\em $\mathrm{LE}$-logic} we understand any axiomatic extension of $\mathbf{L}_\mathrm{LE}$ in the language $\mathcal{L}_{\mathrm{LE}}$. If all the axioms in the extension are analytic inductive (cf.~\cite[Definition 55]{greco2018unified}) we say that the given $\mathrm{LE}$-logic is {\em analytic}. 
  \paragraph{LE-algebras.}
		For any tuple $(\mathcal{F}, \mathcal{G})$ of disjoint sets of function symbols as above, a {\em  lattice expansion} (abbreviated as LE) is a tuple $\bba = (\bbL, \mathcal{F}^\bbA, \mathcal{G}^\bbA)$ such that $\bbL$ is a lattice, $\mathcal{F}^\bbA = \{f^\bbA\mid f\in \mathcal{F}\}$ and $\mathcal{G}^\bbA = \{g^\bbA\mid g\in \mathcal{G}\}$, such that every $f^\bbA\in\mathcal{F}^\bbA$ (resp.\ $g^\bbA\in\mathcal{G}^\bbA$) is an $n_f$-ary (resp.\ $n_g$-ary) operation on $\bbA$. We will often simplify notation and write e.g.\ $f$ for $f^\bbA$.
  Such an operation $f$ (resp.~$g$)  is an {\em operator} if for every $1\leq i\leq n$,  
$$f(\vp)[q\vee^{\epsilon_i} r]_i = f(\vp)[q]_i \vee f(\vp)[r]_i \quad \text{ and }\quad
g(\vp)[q\wedge^{\epsilon_i} r]_i = g(\vp)[q]_i \wedge g(\vp)[r]_i,$$
and it is {\em normal} if
$$f(\vp)[\bot^{\epsilon_i}]_i = \bot \quad \text{ and }\quad g(\vp)[\top^{\epsilon_i}]_i = \top.$$
More concisely, in a normal LE $\bba$, each operation $f^\bba\in \mathcal{F}^\bbA$ (resp.\ $g^\bba\in \mathcal{G}^\bbA$) is finitely join-preserving (resp.\ meet-preserving) in each coordinate when regarded as a map $f^\bba: \bba^{\varepsilon_f}\to \bba$ (resp.\ $g^\bba: \bba^{\varepsilon_g}\to \bba$).
A normal LE as above is {\em complete} if, in addition, $\mathbb{L}$ is a complete lattice and the operation corresponding to each $f\in \mathcal{F}$ (resp.~$g\in \mathcal{G}$) is coordinate-wise completely join-preserving (resp.~meet-preserving) when regarded as a map $f^\bba: \bba^{\varepsilon_f}\to \bba$ (resp.\ $g^\bba: \bba^{\varepsilon_g}\to \bba$). By well known order-theoretic facts (cf.~\cite[Proposition 7.34]{DaveyPriestley02}), a complete normal LE is also {\em completely residuated}, i.e.~the right (resp.~left) residuals $f^\sharp_i$ (resp.~$g^\flat_i$) in each coordinate $i$ exist of the operations corresponding to  every $f\in \mathcal{F}$ (resp.~$g\in \mathcal{G}$).
Let $\mathbb{LE}$ be the class of LEs. Sometimes we will refer to certain LEs as $\mathcal{L}_\mathrm{LE}$-algebras when we wish to emphasize that these algebras have a compatible signature with the logical language we have fixed.

 Henceforth, every LE is assumed to be normal, so the adjective `normal' will be typically dropped. The class of all LEs is equational, and can be axiomatized by the usual lattice identities, and the identities requiring that every operation $f\in \mathcal{F}$ (resp.~$g\in \mathcal{G}$) is coordinate-wise finitely join-preserving (resp.~meet-preserving) w.r.t.~their associated order-type.
\paragraph{Canonical extensions of normal LEs.} \label{def:can:ext}
The \emph{canonical extension} of a lattice $L$ is a complete lattice $L^\delta$ with $L$ as a sublattice, satisfying
\emph{denseness}: every element of $L^\delta$ can be expressed both as a join of meets and as a meet of joins of elements from $L$, and
	\emph{compactness}: for all $S,T \subseteq L$, if $\bigwedge S \leq \bigvee T$ in $L^\delta$, then $\bigwedge F \leq \bigvee G$ for some finite sets $F \subseteq S$ and $G\subseteq T$.
It is well known that  $L^\delta$ is unique up to isomorphism fixing $L$ (cf.\ e.g.\ \cite[Section 2.2]{GNV05}), and that $L^\delta$ is a  complete  lattice. 
If $L$ is bounded we assume that the bounds of $L$ and $L^\delta$ coincide (cf. Remark 2.9 in \cite{gehrke2001bounded}).
The {\em canonical extension} of an
$\mathcal{L}_\mathrm{LE}$-algebra $\bbA = (L, \mathcal{F}^\bbA, \mathcal{G}^\bbA)$ is the perfect  $\mathcal{L}_\mathrm{LE}$-algebra
$\bbA^\delta: = (L^\delta, \mathcal{F}^{\bbA^\delta}, \mathcal{G}^{\bbA^\delta})$ such that $f^{\bbA^\delta}$ and $g^{\bbA^\delta}$ are defined as the
$\sigma$-extension of $f^{\bbA}$ and as the $\pi$-extension of $g^{\bbA}$ respectively, for all $f\in \mathcal{F}$ and $g\in \mathcal{G}$ (cf.\ \cite{sofronie2000duality, sofronie2000duality2}). It is well known (cf.~\cite{gehrke2001bounded}) that the $\sigma$-extension (resp.~$\pi$-extension) of a map $f$ (resp.~$g$) which is finitely $\varepsilon$-join-preserving (resp.~$\varepsilon$-meet-preserving) for some order-type $\varepsilon$ is completely $\varepsilon$-join-preserving (resp.~$\varepsilon$-meet-preserving). Hence, the canonical extension of an LE is a complete, hence fully residuated, LE.
\paragraph{Algebraic semantics of LE-logics.}	Each language $\mathcal{L}_\mathrm{LE}$ is interpreted in the appropriate class of LEs by considering the unique homomorphic extensions of assignments of proposition variables.

	For every LE $\bba$, the symbol $\vdash$ in  sequents $\varphi\vdash \psi$ is interpreted as the lattice order $\leq$. That is, sequent $\varphi\vdash\psi$ is valid in $\bba$ if $h(\varphi)\leq h(\psi)$ for every homomorphism $h$ from the $\mathcal{L}_\mathrm{LE}$-algebra of formulas over $\mathsf{AtProp}$ to $\bba$. The notation $\mathbb{LE}\models\varphi\vdash\psi$ indicates that $\varphi\vdash\psi$ is valid in every LE. Then it is easy to verify by inspecting the rules that the minimal LE-logic $\mathbf{L}_\mathrm{LE}$ 
 is sound w.r.t.~its corresponding class of algebras $\mathbb{LE}$. Moreover, by means of a routine Lindenbaum-Tarski construction, it can be shown that the minimal LE-logic  is also complete with respect to $\mathbb{LE}$-algebras, i.e.\ that any sequent $\varphi\vdash\psi$ is provable in $\mathbf{L}_\mathrm{LE}$ iff $\mathbb{LE}\models\varphi\vdash\psi$. 
	
\subsection{$\mathrm{LE}$-frames and their complex algebras} \label{sec:leframes}
In this section we recall the definition of $\mathrm{LE}$-frames, a relational semantic environment that acts as the dual of lattice expansions. We introduce notational conventions, then $\mathrm{LE}$-frames and their properties, and finally we show how to obtain a lattice expansion from an  $\mathrm{LE}$-frame.

From now on, we fix an arbitrary  LE-signature $\mathcal{L}=\mathcal{L}(\mathcal{F}, \mathcal{G})$.


\paragraph{Notation.}\label{ssec:notation}
Let $B_0, B_1$ be sets and $S \subseteq B_0 \times B_1$ a binary relation. For subsets $X_0 \subseteq B_0$ and $X_1 \subseteq B_1$ define
$$S^{(0)}[X_1] = \{x_0 \in B_0\mid \forall x_1(x_1 \in X_1 \Rightarrow x_0Sx_1) \},$$
$$S^{(1)}[X_0] = \{x_1 \in B_1\mid \forall x_0(x_0 \in X_0 \Rightarrow x_0Sx_1)\}.$$
If the relation $S$ is fixed, we also use the shorter notation $X_1^\downarrow$ for $S^{(0)}[X_1]$ and $X_0^\uparrow$ for $S^{(1)}[X_0]$.

We now generalize these operations to relations of higher arity.
For a sequence of elements $\overline x=(x_0,\ldots,x_n)$ and a sequence of sets $\overline{X} = (X_0,\ldots, X_n)$ we write $\overline{x}\in\overline{X}$
to indicate that $x_i\in X_i$ for $0\le i\le n$. Notice that such sequences can have length $1$. We let 
$$
\bari{x} = (x_0,\ldots,x_{i-1}, x_{i+1},\ldots, x_n)\quad\text{ and }\quad
\bari{X} = (X_0,\ldots,X_{i-1}, X_{i+1},\ldots, X_n),
$$
\[\bari{x}(w): = (x_0,\ldots,x_{i-1}, w, x_{i+1},\ldots, x_n), \quad\text{ and }\quad \bari{X}(B): = (X_0,\ldots,X_{i-1}, B, X_{i+1},\ldots, X_n).\]
For sets $B_0,\ldots B_n$, an $n{+}1$-ary relation $S\subseteq B_0\times \cdots\times B_n$, and subsets $X_j\subseteq B_j$ for $0\leq j\le n$, define
$$
S^{(i)}[\bari{X}] = \{x_i \in B_i\mid  \forall \bari{x}
(\bari{x}\in\bari{X} \Rightarrow \bari{x}(x_i)\in S) \},
$$
So, for example, if $\overline{X}$ is s.t.~$X_j\subseteq B_j$ for $1\le j\le n$ and $\overline{x}\in \overline{X}$, then  $$S^{(0)}[\overline{X}]=\{x_0\in B_0\mid \forall \overline{x}(\overline{x}\in \overline{X}\Rightarrow (x_0,\ldots,x_n)\in S)\}.$$ Throught this article, to simplify the notation, when each $X_i=\{x_i\}$, we will write $S^{0}[\overline{x}]$, $S^{(i)}[\bari{x}]$, $x^{\uparrow}$ and $x^{\downarrow}$  instead of $S^{0}[\overline{X}]$, $S^{(i)}[\bari{X}]$, $\{x\}^{\uparrow}$ and $\{x\}^{\downarrow}$ respectively.

\begin{lemma}\label{lem:adjunction}
If $S \subseteq B_0 \times \cdots\times B_n$, $\overline X=(X_1,\ldots,X_n)$ and $X_j\subseteq B_j$ for $1\le j\le n$,  then for any $1\leq i \leq n$,
\begin{enumerate}
\item $X_0\subseteq S^{(0)}[\overline{X}]\quad \text{ iff }\quad X_i \subseteq S^{(i)}[X_0, \bari{X}]$.
\item $X_i \subseteq S^{(i)}[S^{(0)}[\overline{X}], \bari{X}].$
\end{enumerate}
\end{lemma}
\begin{proof}
1. 
\begin{center}
\begin{tabular}{rll}
  &$X_0\subseteq S^{(0)}[\overline{X}]$&  \\
$\Leftrightarrow$  &$\forall x_0(x_0\in X_0\Rightarrow x_0\in S^{(0)}[\overline{X}])$& \fns def.~of $\subseteq$\\
$\Leftrightarrow$ &$\forall x_0(x_0\in X_0\Rightarrow \forall\overline{x}(\overline{x}\in\overline{X}\Rightarrow (x_0,\overline{x})\in S))$& \fns def.~of $S^{(0)}[-]$\\
$\Leftrightarrow$ &$\forall x_0,\overline{x}(x_0\in X_0\ \&\ \overline{x}\in\overline{X}\Rightarrow (x_0,\overline{x})\in S)$& \fns quantifier equivalence\\
$\Leftrightarrow$ &$\forall x_0,x_i,\overline{x}(x_0\in X_0\ \&\ x_i\in X_i\ \&\ \bari{x}\in\bari{X}\Rightarrow (x_0,\overline{x})\in S)$& \fns quantifier equivalence\\
$\Leftrightarrow$ &$\forall x_i(x_i\in X_i\ \Rightarrow \forall x_0,\bari{x}(x_0\in X_0\ \&\ \bari{x}\in\bari{X}\Rightarrow (x_0,\overline{x})\in S))$& \fns quantifier equivalence\\
$\Leftrightarrow$  &$\forall x_i(x_i\in X_i\Rightarrow x_i\in S^{(i)}[X_0,\bari{X}])$& \fns def.~of $S^{(i)}[-]$\\
$\Leftrightarrow$ &$X_i \subseteq S^{(i)}[X_0, \bari{X}].$& \fns def.~of $\subseteq$
\end{tabular}
\end{center}
2. 
\begin{center}
\begin{tabular}{rll}
  &$S^{(0)}[\overline{X}]\subseteq S^{(0)}[\overline{X}]$&  \\
$\Leftrightarrow$ &$X_i \subseteq S^{(i)}[S^{(0)}[\overline{X}], \bari{X}].$& item 1.
\end{tabular}
\end{center}
\end{proof}

\paragraph{$\mathrm{LE}$-frames.}
A {\em polarity} is a structure $\mathbb{W} = (W, W^\partial, N)$ where $W$ and $W^\partial$ are arbitrary sets and $N\subseteq W\times W^\partial$ is a binary relation (the notation $W^\partial$ is convenient for subsequent definitions). The relation $N$ defines
a closure operator $\gamma_N:\mathcal{P}(W) \rightarrow \mathcal{P}(W)$ (resp.~$\gamma_N:\mathcal{P}(W^\partial) \rightarrow \mathcal{P}(W^\partial)$)
defined by $\gamma_N(X)=X^{\uparrow\downarrow}$ (resp.~$\gamma_N(Y)=Y^{\downarrow\uparrow}$). A subset $X\subseteq W$ (resp.~$Y\subseteq W^\partial$) is 
\emph{Galois stable} if $\gamma_N(X)=X$ (resp.~$\gamma_N(Y)=Y$).

\begin{definition}\label{def: LE frame}
An {\em $\mathcal{L}$-frame} is a tuple $\mathbb{F} = (\mathbb{W}, \mathcal{R}_{\mathcal{F}}, \mathcal{R}_{\mathcal{G}})$ such that $\mathbb{W} = (W, W^\partial,  N)$ is a polarity, $\mathcal{R}_{\mathcal{F}} = \{R_f\mid f\in \mathcal{F}\}$, and $\mathcal{R}_{\mathcal{G}} = \{R_g\mid g\in \mathcal{G}\}$ such that  for each $f\in \mathcal{F}$ and $g\in \mathcal{G}$, the symbols $R_f$ and  $R_g$ respectively denote $n_f{+}1$-ary and $n_g{+}1$-ary relations on $\mathbb{W}$,
$$
R_f \subseteq W^\partial \times W^{\epsilon_f}    \ \mbox{ and }\ R_g \subseteq W \times (W^\partial)^{\epsilon_g}. 
$$
In addition, we assume that 
the relations $R_f$ and $R_g$ be $N$-{\em compatible}, i.e.~the following sets are Galois-stable (from now on abbreviated as {\em stable}) for all $w_0 \in W$, $v_0 \in W^\partial$, $\overline{w} \in W^{\epsilon_f}$, and $\overline{v} \in (W^\partial)^{\epsilon_g}$ and $1\leq i\leq n_f$ (resp.~$1\leq i\leq n_g$):

$$
R_f^{(0)}[\overline{w}],\qquad R_f^{(i)}[v_0, \bari{w}],\qquad
R_g^{(0)}[\overline{v}],\quad\text{and}\quad R_g^{(i)}[w_0, \bari{v}].$$
\end{definition}
In what follows, for any order type $\epsilon$ on $n$,  we let
$\overline{X}= (X_1, \ldots, X_n)$,
where $X_i \subseteq W^{\epsilon_i}$ for all $1 \leq i \leq n$.
Likewise, we let
$\overline{Y}= (Y_1, \ldots, Y_n)$ where $Y_i \subseteq (W^\partial)^{\epsilon_i}$ for all $1 \leq i \leq n$.   Also, for a set $X\subseteq W$ (resp.~$Y\subseteq W^\partial$), we let $X^1: = X$ and  $X^\partial:=X^\uparrow$ (resp.~$Y^1: = Y$ and  $Y^\partial:=Y^\downarrow$), and remind the reader that $X^\partial$ and $Y^\partial$ are  stable subsets. Finally, we let $N^1: = N$ and let $N^\partial$ denote the converse of $N$.



\begin{lemma}\label{lem:Ri closed}
For any $\mathcal{L}$-frame $\mathbb{F} = (\mathbb{W}, \mathcal{R}_{\mathcal{F}}, \mathcal{R}_{\mathcal{G}})$, any $f \in \mathcal{F}$ and any $g \in \mathcal{G},$ \begin{enumerate}
\item 
if $Y_0 \subseteq W^\partial$,
then
$R_f^{(0)}[\overline{X}] $ and $R_f^{(i)}[Y_0, \bari{X}] $ are stable sets for all $1 \leq i \leq n_f$;

\item 
if $X_0 \subseteq W$,
then
$R_g^{(0)}[\overline{Y}] $ and $R_g^{(i)}[X_0, \bari{Y}] $ are stable sets for all $1 \leq i \leq n_g$.
\item $R_f^{(0)}[\bari{X}(X_i)] = R_f^{(0)}[\bari{X}(\gamma_N(X_i))]$ for all $1 \leq i \leq n_f$;
\item $R_g^{(0)}[\bari{Y}(Y_i)] = R_g^{(0)}[\bari{Y}(\gamma_N(Y_i))]$ for all $1 \leq i \leq n_g$.
\end{enumerate}
\end{lemma}

\begin{proof}
1. By definition, $R_{f}^{(0)}[\overline{X}] =  \bigcap_{\overline{w} \in \overline{X}} R_{f}^{(0)}[\overline{w}]$ and  $R_{f}^{(i)}[Y_0, \bari{X}] = \bigcap_{u \in Y_0, \bari{w} \in \bari{X}} R_{f}^{(i)}[u, \bari{w}]$ are stable, being intersections of stable sets. 
The proof of item 2 is analogous.\\
3. \begin{center}
\begin{tabular}{rcll}
  $u \in R_f^{(0)}[\overline{X}]$& 
iff & $X_i \subseteq R_{f}^{(i)}[\{u\}, \bari{X}]$ & Lemma \ref{lem:adjunction}.1\\
& iff &$\gamma_N(X_i) \subseteq R_{f}^{(i)}[u, \bari{X}]$&item 1\\
& iff &$u \in R_f^{(0)}[\bari{X}(\gamma_N(X_i))]$&Lemma \ref{lem:adjunction}.1
\end{tabular}
\end{center}
The proof of item 4 is analogous.
\end{proof}

For the running Example~\ref{ex:1} we have $\mathcal L$-frames $\mathbb{F} = (W,W^\partial,N, \{R_\mand\}, \{R_\Box ,R_\backslash\})$, where $R_\mand\subseteq W^\partial\times W^2$, $R_\Box\subseteq W\times W^\partial$, $R_\backslash\subseteq W^2\times W^\partial$ and each of these relations is $N$-compatible. Since $\backslash$ is the left residual of $\mand$, the relations $R_{\mand}$ and $R_{\backslash}$ are interdefinable via $R_\mand(u,v,w)\iff R_{\backslash}(w,v,u)$.

\paragraph{Complex algebras of $\mathrm{LE}$-frames.}
\label{sec:ComplexAlgebrasOfLE-frames}
For a polarity $\mathbb{W}$, we let $\mathbb{W}^+$ be the complete $\bigcap$-semilattice of all Galois-stable sets of $\gamma_N$. As is well known, $\mathbb{W}^+$ is a complete lattice, in which $\bigvee \mathcal S := \gamma_N(\bigcup\mathcal S)$ for any $\mathcal S \subseteq \mathbb W^+$. Moreover, $\mathbb{W}^+$ can be equivalently obtained as the dual lattice of the Galois-stable sets of the closure operator $\gamma^\partial_N: \mathcal{P}(W^\partial) \rightarrow \mathcal{P}(W^\partial)$ defined by $\gamma^\partial_N(Y)= Y^{\downarrow\uparrow}$. 

For a lattice $\mathbb{L}$, let $\mathbb{W}_L$ be the polarity $(L, L, \leq)$.
In this case, the complete lattice $\mathbb W_L^+$ is the MacNeille completion of $L$.

\begin{definition} \label{def:complexalg}
The {\em complex algebra} of  an $\mathcal{L}$-frame $\mathbb{F} = (\mathbb{W}, \mathcal{R}_{\mathcal{F}}, \mathcal{R}_{\mathcal{G}})$ is the algebra
$$\mathbb{F}^+ = (\mathbb{L}, \{f_{R_f} \mid f \in \mathcal{F}\}, \{g_{R_g} \mid g \in \mathcal{G}\}),$$ 
where $\mathbb{L} := \mathbb{W}^+$ is the complete lattice of stable sets of $W$, and for $f \in \mathcal{F}$ and $g \in \mathcal{G}$,
$ f_{R_f}: \mathbb{L}^{n}\to \mathbb{L}$ is defined by 
$$
f_{R_f}(X_1,\ldots,X_n) = (R_f^{(0)}[X_1^{\epsilon_{f,1}},\ldots,X_n^{\epsilon_{f,n}}])^\downarrow
$$
and $g_{R_g}: \mathbb{L}^{n_g}\to \mathbb{L}$ is defined by  
$$
g_{R_g}(X_1,\ldots,X_n) = R_g^{(0)}[X_1^{\epsilon^{\partial}_{g,1}},\ldots,X_n^{\epsilon^{\partial}_{g,n}}].
$$
\end{definition}

\begin{prop}\label{prop:F plus is L star algebra}
If $\mathbb{F}$ is an $\mathcal{L}$-frame,  then $\mathbb{F}^+$ is a complete $\mathcal{L}$-algebra.
\end{prop}
\begin{proof}
We need to prove that for every $f \in \mathcal{F}$ and every $g \in \mathcal{G}$, $f_{R_f}$ is a complete $\epsilon_f$-operator and $g_{R_g}$ is a complete $\epsilon_g$-dual operator. Since the underlying lattice of $\mathbb{F}^+$ is complete, it is enough to show that the residuals of every $f \in \mathcal{F}$ and $g \in \mathcal{G}$ in each coordinate exist (cf.~\cite[Proposition 7.34]{DaveyPriestley02}).

Let $f \in \mathcal{F}$, $X_0,\ldots,X_n,\in  \gamma_N[\mathcal{P}(W)]$ and define $$(f_{R_f})^{\sharp}_{i}(X_1,\ldots,X_n)=(R^{(i)}_f[X_i^{\epsilon^\partial_{f,i}},(\bari{X})^{\epsilon_f}])^{\epsilon_{f,i}}.$$ Notice that, by Lemma \ref{lem:Ri closed}.1, $R^{(i)}_f[X_i,\bari{X}]$ is stable and hence $(f_{R_f})^{\sharp}_{i}$ is well defined. We have:
\begin{center}
\begin{tabular}{cll}
& $f_{R_f}(X_1,\ldots,X_n) \subseteq X_0$ & \\
$\Leftrightarrow$ & $(R_f^{(0)}[X_1^{\epsilon_{f,1}},\ldots,X_n^{\epsilon_{f,n}}])^\downarrow\subseteq X_0$ & definition of $f_{R_f}$\\
$\Leftrightarrow$ & $(X_0)^
\uparrow\subseteq (R_f^{(0)}[X_1^{\epsilon_{f,1}},\ldots,X_n^{\epsilon_{f,n}}])^{\downarrow\uparrow}$ & $(-)^{\uparrow}$ is antitone\\
$\Leftrightarrow$ & $(X_0)^{\uparrow}\subseteq R_f^{(0)}[X_1^{\epsilon_{f,1}},\ldots,X_n^{\epsilon_{f,n}}]$ & Lemma \ref{lem:Ri closed}.1\\
$\Leftrightarrow$ & $X^{\epsilon_{f,i}}_i
\subseteq R_f^{(i)}[(X_0)^{\uparrow},(\bari{X})^{\epsilon_f}]$. & Lemma \ref{lem:adjunction}.1\\
\end{tabular}
\end{center}
If $\epsilon_{f,i}=1$, the last inequality is tantamount to $X_i\subseteq (f_{R_f})^{\sharp}_i(\bari{X}(X_0))$. If $\epsilon_{f,i}=\partial$, the last inequality implies that $(R_f^{(i)}[(X_0)^{\uparrow},(\bari{X})^{\epsilon_f}])^{\downarrow}\subseteq (X_i)^{\uparrow\downarrow}$, which is equivalent to $(f_{R_f})^{\sharp}_i(\bari{X}(X_0))\subseteq X_i$, since $X_i\in\gamma_N[\mathcal{P}(W)]$. In both cases $(f_{R_f})^{\sharp}_i$, is the $i$-th residual of $f_{R_f}$. The argument for $g\in\mathcal{G}$ is analogous.
\end{proof}

\subsection{Proper display calculi for basic normal LE-logics} \label{ssec:syntactic frames associated w algebras}

In this section we  recall the definition of 
the proper display calculus $\mathrm{D.LE}$ for the basic normal $\mathcal{L}$-logic and its cut-free counterpart $\mathrm{cfD.LE}$ for a fixed but arbitrary LE-signature $\mathcal{L} = \mathcal{L}(\mathcal{F}, \mathcal{G})$  (cf.~Section \ref{subset:language:algsemantics}.1). 
Our presentation is a more streamlined version of the one introduced in \cite{greco2018unified} for DLE-logics and then straightforwardly generalized to LE-logics in \cite{ChnGrePalTzi21}. 

The syntax in display calculi is two-layered. One layer, consisting of formulas, cannot be manipulated, while the second, consisting of structures can. Hence, given $\mathcal{L}(\mathcal{F}, \mathcal{G})$ the language of its display calculus is augmented with structural symbols each corresponding to operations in $\mathcal{L}(\mathcal{F}^*, \mathcal{G}^\ast)$. While formulas can appear freely in sequents of the display calculus, structures are restricted on where they can appear in a sequent.

Let $S_{\mathcal{F}} \ceq \{\FH \mid f\in \mathcal{F}^*\}$ and $S_{\mathcal{G}} \ceq \{\GC \mid g\in \mathcal{G}^*\}$ be the sets of structural connectives associated with  $\mathcal{F}^*$ and $ \mathcal{G}^*$ respectively, where $\mathcal{L}(\mathcal{F}^*, \mathcal{G}^\ast)$ denotes the fully residuated language-expansion of $\mathcal{L}(\mathcal{F}, \mathcal{G})$. Each such structural connective has the same arity and  order-type  of its associated  $f\in \mathcal{F}^\ast$ (resp.~$g\in \mathcal{G}^\ast$).

The calculus $\mathrm{D.LE}$ manipulates sequents $\Gamma \Rightarrow \Delta$, where 
$\Gamma$ and $\Delta$ are structures of two sorts which are built from formulas and are defined by the following simultaneous recursions: 

\begin{center}
\begin{tabular}{@{}r@{}l@{}}
$\rule[-1.2ex]{0pt}{0ex}\mathsf{Fm} \ni \varphi$ \ & $ \cceq \ p \mid \bot \mid \top \mid  \varphi \vee \varphi \mid \varphi \wedge \varphi \mid f (\ophi) \mid g (\ophi) 
$ \\


\rule[-1.2ex]{0pt}{0ex}$\mathsf{Str}_\mathcal{F} \ni \Gamma$ \ & $ \cceq \  \varphi \mid \AATOP \mid \FH\, (\OSI) 
$ \\

\rule[-1.2ex]{0pt}{0ex}$\mathsf{Str}_\mathcal{G} \ni \Delta$ \ & $ ::=\  \varphi \mid \ABOT \mid \GC\, (\OSI)$ \\
\end{tabular}
\end{center}

\noindent where $p$ is an atomic formula and $f \in \mathcal{F}$ and $g \in \mathcal{G}$ and $\FH 
\in S_{\mathcal{F}}$ and $\GC
\in S_{\mathcal{G}}$, 
and $\overline{\Sigma}\in \mathsf{Str}_\mathcal{F}^{\epsilon_f}$ (resp.~$\overline{\Sigma}\in \mathsf{Str}_\mathcal{G}^{\epsilon_g}$) for any  $\overline{\Sigma}$ in the argument of $\FH$ (resp.~$\overline{\Sigma}$ in the argument of $\GC(\overline{\Sigma})$). Here $\mathsf{Str}_\mathcal{F}^{\epsilon_f}$ is defined as $\mathsf{Str}_\mathcal{F}^{\epsilon_1}\times\cdots\times \mathsf{Str}_\mathcal{F}^{\epsilon_n}$ where $\mathsf{Str}_\mathcal{F}^\partial=\mathsf{Str}_\mathcal{G}$, and dually for $\mathsf{Str}_\mathcal{G}^{\epsilon_g}$.

In what follows, for every $K \in S_{\mathcal{F}} \cup S_{\mathcal{G}}$ we use $K(\OSI)[\Gamma]_i$ (resp.~$K(\OSI)[\Delta]_i$) to indicate that the structure $\Gamma$ (resp.~$\Delta$) occurs in the $i$-th coordinate of the vector $\OSI$. 

Below, we list the rules of the calculus D.LE.



\begin{itemize}
\item Identity and cut rules:\footnote{Notice that in the display calculi literature, the identity rule is sometimes defined as $\varphi \fCenter \varphi$, where $\varphi$ is an arbitrary, possibly complex, formula. The difference is inessential given that, in any display calculus, $p \fCenter p$ is an instance of $\varphi \fCenter \varphi$, and $\varphi \fCenter \varphi$ is derivable for any formula $\varphi$ whenever $p \fCenter p$ is the Identity rule.}
\end{itemize}
\begin{center}
\begin{tabular}{rl}
\AXC{\phantom{$\Gamma \fCenter \varphi$}}
\LL{\fns Id}
\UI$p \fCenter p$
\DP
 & 
\AX$\Gamma \fCenter \varphi$
\AX$\varphi \fCenter \Delta$
\RL{\fns Cut}
\BI$\Gamma \fCenter \Delta$
\DP
 \\
\end{tabular}
\end{center}


\begin{itemize}		
\item Display postulates for $f\in \mathcal{F}$ and $g\in \mathcal{G}$: for any $1\leq i \leq n_f$ and $1\leq j\leq n_g$,
\end{itemize}
\begin{itemize}
\item[] If $\varepsilon_{f,i} = 1$ and $\varepsilon_{g,j} = 1$,\footnote{The notation $\FH \dashv \FCS_i$ (resp.~$\GHF_j \dashv \GC$) indicates that $\FH$ and $\FCS_i$ (resp.~$\GHF_j$ and $\GC$) are in a {\em residuated pair} and $\FCS_i$ (resp.~$\GHF_j$) is the right residual (resp.~left residual) of $\FH$ (resp.~$\GC$) in the $i$-th coordinate (resp.~$j$-th coordinate).}
\end{itemize}
\begin{center}
\begin{tabular}{@{}c@{}c@{}}
\AXC{$\FH\, (\OSI)[\Gamma]_i \fCenter \Delta$}
\doubleLine
\LL{\fns $\FH \dashv \FCS_i$}
\UIC{$\Gamma \fCenter \FCS_i\, (\OSI)[\Delta]_i$}
\DP
 & \quad
\AXC{$\Gamma \fCenter \GC\, (\OSI)[\Delta]_j$}
\doubleLine
\RL{\fns $\GHF_j \dashv \GC$}
\UIC{$\GHF_j\, (\OSI)[\Gamma]_j \fCenter \Delta$}
\DP \\
\end{tabular}
\end{center}

\begin{itemize}
\item[] If $\varepsilon_{f,i} = \partial$ and $\varepsilon_{g,j} = \partial$,\footnote{The notation $(\GC, \GCF_j)$ (resp.~$(\FH, \FHS_i)$) indicates that $\GC$ and $\GCF_j$ (resp.~$\FH$ and $\FHS_i$) are in a {\em Galois connection} (resp.~{\em dual Galois connection}) and $\GCF_j$ (resp.~$\FHS_i$) is the right residual (resp.~left residual) of $\GC$ (resp.~$\FH$) in the $j$-th coordinate (resp.~$i$-th coordinate).}
\end{itemize}
\begin{center}
\begin{tabular}{@{}c@{}c@{}}					
\AX$\FH\, (\OSI)[\Delta]_i \fCenter \Delta'$
\doubleLine
\LL{\fns $(\FH, \FHS_i)$}
\UI$\FHS_i\, (\OSI)[\Delta']_i \fCenter \Delta$
\DP
 & \quad
\AX$\Gamma' \fCenter \GC\, (\OSI)[\Gamma]_j$
\doubleLine
\RL{\fns $(\GC, \GCF_j)$}
\UI$ \Gamma \fCenter \GCF_j\, (\OSI)[\Gamma']_j$
\DP \\
\end{tabular}
\end{center}

\begin{itemize}
\item Structural rules for lattice connectives:
\end{itemize}
\begin{center}
\begin{tabular}{rl}
\AX$\AATOP \fCenter \Delta$
\LL{\fns $\aatop_W$}
\UI$\Gamma \fCenter \Delta$
\DP
 & 
\AX$\Gamma \fCenter \ABOT$
\RL{\fns $\abot_W$}
\UI$\Gamma \fCenter \Delta$
\DP
\end{tabular}
\end{center}

\begin{itemize}
\item Logical introduction rules for lattice connectives:
\end{itemize}
\begin{center}
\begin{tabular}{rl}
\AX$\AATOP \fCenter \Delta$
\LL{\fns $\aatop_L$}
\UI$\aatop \fCenter \Delta$
\DP
 \ 
\AXC{$\phantom{\AATOP \fCenter \Delta}$}
\RL{\fns $\aatop_R$}
\UI$\AATOP \fCenter \aatop$
\DP
 & 
\AXC{$\phantom{\Gamma \fCenter \ABOT}$}
\LL{\fns $\abot_L$}
\UI$\abot \fCenter \ABOT$
\DP
 \ 
\AX$\Gamma \fCenter \ABOT$
\RL{\fns $\abot_R$}
\UI$\Gamma \fCenter \abot$
\DP
 \\

 & \\

\AX$\psi \fCenter \Delta$
\LL{\fns $\aand_{L2}$}
\UI$\varphi \aand \psi \fCenter \Delta$
\DP
 \ 
\AX$\varphi \fCenter \Delta$
\LL{\fns $\aand_{L1}$}
\UI$\varphi \aand \psi \fCenter \Delta$
\DP
 & 
\AX$\Gamma \fCenter \varphi$
\AX$\Gamma \fCenter \psi$
\RL{\fns $\aand_R$}
\BI$\Gamma \fCenter \varphi \aand \psi$
\DP
 \\

 & \\

\AX$\varphi \fCenter \Delta$
\AX$\psi \fCenter \Delta$
\LL{\fns $\aor_L$}
\BI$\varphi \aor \psi \fCenter \Delta$
\DP
 & 
\AX$\Gamma \fCenter \varphi$
\RL{\fns $\aor_{R1}$}
\UI$\Gamma \fCenter \varphi \aor \psi$
\DP
 \ 
\AX$\Gamma \fCenter \psi$
\RL{\fns $\aor_{R2}$}
\UI$\Gamma \fCenter \varphi \aor \psi$
\DP
 \\
\end{tabular}

\end{center}

\begin{itemize}
\item Logical introduction rules for $f\in\mathcal{F}$ and $g\in\mathcal{G}$:
\end{itemize}
				\begin{center}
					\begin{tabular}{c c}
						\bottomAlignProof
						\AX$\FH\, (\ophi) \fCenter \Delta$
						\LL{\fns$f_L$}
						\UI$f(\ophi) \fCenter \Delta$
						\DP
						&
						\bottomAlignProof
						\AX$\Gamma \fCenter \GC\, (\ophi)$
						\RL{\fns$g_R$}
						\UI$\Gamma \fCenter g(\ophi)$
						\DP
					\end{tabular}
                \end{center}
                \begin{center}
					\begin{tabular}{ccc}
						\bottomAlignProof
						\AxiomC{$\Big(\Sigma_i \fCenter^{\!\!\epsilon_{f,i}}\, \varphi_i  \mid 1\leq i\leq n_f\Big)$}
						\RL{\fns$f_R$}
						\UI$\FH\, (\OSI)\fCenter f(\ophi)$
						\DP
					\end{tabular}
                    \end{center}
where  $\Sigma_i {\fCenter^{\!\!\epsilon_{f,i}}}\, \varphi_i$ is $\Sigma_i\fCenter \varphi_i$  (resp.~$\varphi_i\fCenter \Sigma_i$) if $\epsilon_{f,i}=1$ (resp.~$\epsilon_{f,i}=\partial$). 
                    \begin{center}
					\begin{tabular}{c c c }
						\bottomAlignProof
						\AxiomC{$\Big(\varphi_i \fCenter^{\!\!\epsilon_{{g,i}}}\; \Sigma_i \,\mid\, 1\leq i\leq n_g \Big)$}
						\LL{\fns$g_L$}
						\UI$g(\ophi) \fCenter \GC\, (\OSI)$
						\DP
					\end{tabular}
     \end{center}
     where  $\varphi_i \fCenter^{\!\!\epsilon_{g,i}}\; \Sigma_i$ is $\varphi_i\fCenter \Sigma_i$  (resp.~$\Sigma_i\fCenter \varphi_i$) if $\epsilon_{g,i}=1$ (resp.~$\epsilon_{g,i}=\partial$).

				In particular, if $f$ and $g$ are $0$-ary (i.e.~they are constants), the rules $f_R$ and $g_L$ above reduce to the axioms ($0$-ary rules) $\FH \fCenter f$ and $g \fCenter \GC$.

Let $\mathrm{cfD.LE}$ (resp.~$\mathrm{cfD.LE}$) denote the calculus obtained by removing Cut in $\mathrm{D.LE}$. In what follows, we indicate that the sequent $\varphi \Rightarrow \psi$ is derivable in $\mathrm{D.LE}$ (resp.~in $\mathrm{cfD.LE}$) by $\vdash_{\mathrm{D.LE}} \varphi \Rightarrow \psi$ (resp.~by $\vdash_{\mathrm{cfD.LE}} \varphi \Rightarrow \psi$).

\begin{thm}(\cite[Section 4.2]{greco2018unified})\label{prop:soundness of D.LE}
The calculus $\mathrm{D.LE}$ (and hence also $\mathrm{cfD.LE}$) is sound and complete with respect to the class of  complete $\mathcal{L}$-algebras. Furthermore the calculus $\mathrm{D.LE}$ is conservative w.r.t.\ the corresponding $\mathrm{LE}$-logic.
\end{thm}

In the presentation of the language of the calculus above we use $\Gamma_1, \Gamma_2, \ldots$ and $\Delta_1, \Delta_2, \ldots$ as meta-variables for structures. To formally present analytic structural rules, we need to utilize meta-structures, i.e., structures that are constructed by structural meta-variables. In what follows, we will introduce explicitly a language of meta-variables and meta-terms, which will be useful in the remainder of this paper. Let $\mathsf{MVar} = \mathsf{MVar}_{\mathcal{F}}\uplus \mathsf{MVar}_{\mathcal{G}}$ be the denumerable set of {\em meta-variables} of sorts $\Gamma_1, \Gamma_2, \ldots\in \mathsf{MVar}_{\mathcal{F}}$ and $\Delta_1, \Delta_2, \ldots \in \mathsf{MVar}_{\mathcal{G}}$.
The sets $\mathsf{MStr}_\mathcal{F}$ and $\mathsf{MStr}_\mathcal{G}$ of the $\mathcal{F}$- and $\mathcal{G}$-{\em meta-structures} are defined by simultaneous induction as follows:

\begin{center}
\begin{tabular}{c}
$\mathsf{MStr}_\mathcal{F} \ni S \ \cceq \  \Gamma \mid \FH(\overline{S})$\\

$\mathsf{MStr}_\mathcal{G} \ni T \ \cceq \  \Delta \mid \GC(\overline{T})$
\end{tabular}
\end{center}
where $\Gamma\in \mathsf{MVar}_{\mathcal{F}}$, $\Delta\in \mathsf{MVar}_{\mathcal{G}}$, $\FH \in \mathcal{F}^\ast$ and $\GC \in \mathcal{G}^\ast$ and $\overline{S} \in \mathsf{MStr}_\mathcal{F}^{\epsilon_f}$,  and $\overline{T} \in \mathsf{MStr}_\mathcal{G}^{\epsilon_g}$, and for any order type $\epsilon$ on $n$, we let $\mathsf{MStr}_\mathcal{F}^{\epsilon} \ceq \prod_{i = 1}^{n}\mathsf{MStr}_\mathcal{F}^{\epsilon_i}$ and $\mathsf{MStr}_\mathcal{G}^{\epsilon} \ceq \prod_{i = 1}^{n}\mathsf{MStr}_\mathcal{G}^{\epsilon_i}$, where for all $1 \leq i \leq n$,

\begin{center}
\begin{tabular}{ll}
$\mathsf{MStr}_\mathcal{F}^{\epsilon_i} = \begin{cases} 
\mathsf{MStr}_\mathcal{F} &\mbox{ if } \epsilon_i = 1\\
\mathsf{MStr}_\mathcal{G} &\mbox{ if } \epsilon_i = \partial
\end{cases}\quad$
&
$\mathsf{MStr}_\mathcal{G}^{\epsilon_i} = \begin{cases}
\mathsf{MStr}_\mathcal{G}& \mbox{ if } \epsilon_i = 1,\\
\mathsf{MStr}_\mathcal{F} & \mbox{ if } \epsilon_i = \partial.
\end{cases}$
\end{tabular}
\end{center}

An analytic structural rule is formally represented by a rule of the form
\begin{center}
\AXC{$(S_1\Rightarrow T_1)[\OSI_1]$}
\AXC{$\cdots$}
\AXC{$(S_n\Rightarrow T_n)[\OSI_n]$}
\TIC{$(S_0 \Rightarrow T_0)[\OSI_0]$}
\DP
\end{center}
where $\OSI_i$ with $0 \leq i \leq n$ is the set of meta-variables occurring in each sequent $S_i \Rightarrow T_i$, $\OSI_0 \supseteq \OSI_1 \cup \ldots \cup \OSI_n$, and while structural meta-variables might occur multiple times in the premises they occur only once in the conclusion. 
An instance of the rule $R$ is obtained from $R$ by uniformly substituting each structural meta-variable $\Gamma\in \mathsf{MVar}_{\mathcal{F}}$ with an element of $\mathsf{Str}_{\mathcal{F}}$ and every structural meta-variable $\Delta\in \mathsf{MVar}_{\mathcal{G}}$ with an element of $\mathsf{Str}_{\mathcal{G}}$. A calculus $\mathcal{D}$ contains the analytic rule $R$ if it contains every instance of $R$.

\begin{example}\label{ex:rule}
 Given the LE-language $\mathcal{L}_\mathrm{LE}(\{\mand\},\{\wbox,\backslash\})$ considered in Example \ref{ex:1}, the following is an analytic structural rule: 
\begin{center}
\AX$\Gamma_1 \fCenter (\Gamma_2 \MAND \Gamma_3) \MRARR \WBOX\Delta$
\UI$\Gamma_1 \MAND \BDIA\Gamma_2 \fCenter \Delta \MLARR  \Gamma_3$
\DP
\end{center}
This rule is such that $\Gamma_1, \Gamma_2$, $\Gamma_3$ and $\Delta$ are meta-variables which range over $\mathsf{Str}_{\mathcal{F}}$, and $\mathsf{Str}_{\mathcal{G}}$ respectively, and $(\Gamma_2 \MAND \Gamma_3) \MRARR \WBOX\Delta$, $\Gamma_1 \MAND \BDIA\Gamma_2$ and $\Delta \MLARR  \Gamma_3$ are meta-terms. 
\end{example}

Given any  $\mathrm{LE}$-logic, $\mathcal{L}$, the class of axioms which correspond to analytic structural rules has been characterized in \cite{CiRa14,greco2018unified,ChnGrePalTzi21}. In Appendix \ref{ap:display} we extensively discuss Belnap's presentation of display calculi \cite{Belnap} via conditions C1-C8, and show how the presentation in this section satisfies those conditions.

\section{Functional $\mathcal{L}$-frames}
\label{sec: functional L-frames}
In this section we introduce functional $\mathcal{L}$-frames, which will be the main  semantic environment for the proof of semantic cut elimination.

\subsection{Functional $\mathcal{L}$-frames}\label{rmk:functional implies compatible} An $\mathcal{L}$-frame is {\em functional} if for all $f\in \mathcal{F}$ (resp.~$g\in \mathcal{G}$), the relation $R_f$ (resp.~$R_g$) is {\em functional}, i.e.\ there exists an $n$-ary partial function $f^*:W^{\epsilon_f}\to W$ (resp.~$g^*:(W^{^\partial})^{\epsilon_g}\to W^\partial$) such that $R_f(u,\overline{w})$ iff $f^*(\overline{w})Nu$ (resp.~$R_g(w,\overline{u})$ iff $wNg^*(\overline{u})$). 

 For every polarity $\mathbb{W} = (W, W^\partial, N)$ and any  $f^*:W^{\epsilon_f}\to W$ (resp.~$g^*:(W^{^\partial})^{\epsilon_g}\to W^\partial$),  $f^\ast$ has a {\em pseudo-residual} in its $i$th coordinate if a map  ${f^\ast}^\sharp_i:W^{\epsilon_{f^\sharp_i}} \to W^{\epsilon_{f,i}}$ (resp.~${g^\ast}^\flat_i: (W^\partial)^{\epsilon_{g^\flat_i}}\to (W^\partial)^{\epsilon_{g,i}}$) exists such that
\[f^\ast(\overline{w})N u \quad \text{ iff }\quad w_iN^{\epsilon_{f,i}}{f^\ast}^\sharp_i(\bari{w}(u))\quad \quad 
wN g^\ast(\overline{u}) \quad \text{ iff }\quad {g^\ast}^\flat_i(\bari{u}(w))N^{\epsilon_{g,i}}u_i.\]
If $f^\ast$ (resp.~$g^\ast$) has pseudo-residuals in each coordinate, the relation $R_f\subseteq W^\partial\times W^{\epsilon_f}$ (resp.~$R_g\subseteq W\times (W^\partial)^{\epsilon_g}$) defined as $R_f(u,\overline{w})$ iff $f^*(\overline{w})Nu$ (resp.~$R_g(w,\overline{u})$ iff $wNg^*(\overline{u})$) is $N$-compatible: indeed, $R_f^{(0)}[\overline{w}] = \{u\in W^\partial\mid f^*(\overline{w})Nu\} = f^*(\overline{w})^{\uparrow}$ which is stable by general properties of Galois connections. Likewise, for every $1\leq i\leq n_f$ (resp.~$1\leq i\leq n_g$), the set $R_f^{(i)}[u, \bari{w}] = \{w_i\in W^{\epsilon_{f,i}}\mid f^*(\bari{w}(w_i))Nu\} = \{w_i\in W^{\epsilon_{f,i}}\mid w_iN^{\epsilon_{f,i}}{f^\ast}^\sharp_i(\bari{w}(u))\} = {f^\ast}^\sharp_i(\bari{w}(u))^{\partial}$ (resp.~$R_g^{(i)}[v, \bari{u}] = {g^\ast}^\flat_i(\bari{u}(v))^{\partial}$) is stable. Therefore, by items 1 and 2 of Lemma \ref{lem:Ri closed}, for any $\overline{X}, \overline{Y}$, the following sets are stable, and moreover, from the identities above it follows that:
\begin{equation}\label{eq:functional relation 2}
R_f^{(0)}[\overline{X}] = \bigcap_{\overline{x}\in \overline{X}} f^\ast(\overline{x})^{\uparrow} = N^{(1)}[f^\ast[\bigcup_{\overline{x}\in \overline{X}} \{\overline{x}\}]] = (f^\ast[\overline{X}])^{\uparrow}\quad \text{ and }\quad R_g^{(0)}[\overline{Y}] = (g^\ast[\overline{Y}])^{\downarrow}. \end{equation}
Likewise,  $R_f^{(i)}[Y, \bari{X}] = ({f^\ast}^\sharp_i[\bari{X}(Y)])^{\partial}$ and $R_g^{(i)}[X, \bari{Y}] = ({g^\ast}^\flat_i[\bari{Y}(X)])^{\partial}$.

 As a consequence of these identities, and of items 3 and 4 of Lemma \ref{lem:Ri closed}, for all $1 \leq i \leq n_f$ and  $1 \leq i \leq n_g$,
 \begin{equation} \label{eq:functional relation 3}(f^\ast[\overline{X}])^{\uparrow} = (f^\ast[\overline{\gamma_N(X)}])^{\uparrow}\quad \text{ and } (g^\ast[\overline{Y}])^{\downarrow}=(g^\ast[\overline{\gamma_N(Y)}])^{\downarrow},\end{equation}
 where the $i$-th component of e.g.\ $\overline{\gamma_N(X)}$ is obtained by applying $\gamma_N$ to the $i$-th component of $\overline{X}$.
A {\em fully residuated} functional $\mathcal{L}$-frame is a structure $\mathbb{F} = (\mathbb{W}, \mathcal{R}_{\mathcal{F}}, \mathcal{R}_{\mathcal{G}})$ as above such that every $R_f\in \mathcal{R}_{\mathcal{F}}$ (resp.~$R_g\in \mathcal{R}_{\mathcal{G}}$) is functional and its corresponding map $f^\ast$ (resp.~$g^\ast$) has pseudo-residuals in each coordinate.

\subsection{Interpretation of meta-structures and sequents on complex algebras of fully residuated functional frames} 
In this section we will show that the interpretation of meta-structures on the complex algebras of functional frames can be approximated by the functions of the frame.  
  
For every fully residuated functional $\mathcal{L}$-frame $\mathbb{F}$, we will identify any assignment $h:\mathsf{MVar}\to \mathbb{F}^+$ with its unique homomorphic extension, and hence write both $h(\Gamma)$ and $h(\Delta)$.

\begin{definition}\label{def:image}
	For any fully residuated functional $\mathcal{L}$-frame $\mathbb{F}$, any $h:\mathsf{MVar}\to \mathbb{F}^+$, any $S\in \mathsf{MStr}_\mathcal{F}$ and $T\in\mathsf{MStr}_\mathcal{G}$, the subsets $h\{S\} \subseteq W$ and $h\{T\} \subseteq W^\partial$ are defined by simultaneous recursion as follows: 
	\begin{enumerate}
		\item $h\{\Gamma\} \ceq h(\Gamma)$ and $h\{\Delta\} = h(\Delta)^\uparrow$;
		\item $h\{\FH(\overline{S})\} \ceq f^\ast[ \overline{h\{S\}}] = \{f^\ast(\overline{x})\mid \overline{x}\in \overline{h\{S\}}\}$ for every $f\in \mathcal{F}$;
  \item $h\{\FHS_i(\overline{S})\} \ceq {f^\ast}^\sharp_i[ \overline{h\{S\}}] = \{{f^\ast}^\sharp_i(\overline{x})\mid \overline{x}\in \overline{h\{S\}}\}$ for every $f\in \mathcal{F}$ s.t.~$\epsilon_{f,i} = \partial$;
  \item $h\{\GHF_i(\overline{T})\} \ceq {g^\ast}^\flat_i[ \overline{h\{T\}}] = \{{g^\ast}^\flat_i(\overline{y})\mid \overline{y}\in \overline{h\{T\}}\}$ for every $g\in \mathcal{G}$ s.t.~$\epsilon_{g,i} = 1$;
		\item $h\{\GC(\overline{T})\} \ceq g^\ast[\overline{h\{T\}}]= \{g^\ast(\overline{y})\mid \overline{y}\in \overline{h\{T\}}\}$ for every $g\in \mathcal{G}$; 
\item $h\{\GCF_i(\overline{T})\} \ceq {g^\ast}^\flat_i[\overline{h\{T\}}]= \{{g^\ast}^\flat_i(\overline{y})\mid \overline{y}\in \overline{h\{T\}}\}$ for every $g\in \mathcal{G}$ s.t.~$\epsilon_{g,i} = \partial$;
\item $h\{\FCS_i(\overline{S})\} \ceq {f^\ast}^\sharp_i[ \overline{h\{S\}}] = \{{f^\ast}^\sharp_i(\overline{x})\mid \overline{x}\in \overline{h\{S\}}\}$ for every $f\in \mathcal{F}$ s.t.~$\epsilon_{f,i} = 1$.
	\end{enumerate}
 where $\overline{S} \subseteq \mathsf{MStr}^{\epsilon_f}_\mathcal{F}$, $\overline{h\{S\}} \ceq \prod_{i = 1}^{n_f}h\{S^{\epsilon_{f,i}}\}$ (resp.~$\overline{T} \subseteq \mathsf{MStr}^{\epsilon_g}_\mathcal{G}$, $\overline{h\{T\}} \ceq \prod_{i=1}^{n_g}h\{T^{\epsilon_{g,i}}\}$), such that
	\begin{center}
 \begin{tabular}{ccc}
	$S^{\epsilon_{f,i}} \in \begin{cases} 
	\mathsf{MStr}_\mathcal{F}&\mbox{ if } \epsilon_{f,i} = 1\\
	\mathsf{MStr}_\mathcal{G} &\mbox{ if } \epsilon_{f,i} = \partial
	\end{cases}$ &$\quad$ &
	
$
	T^{\epsilon_{g,i}} \in \begin{cases} 
	\mathsf{MStr}_\mathcal{G}&\mbox{ if } \epsilon_{f,i} = 1\\
	\mathsf{MStr}_\mathcal{F} &\mbox{ if } \epsilon_{f,i} = \partial.
	\end{cases}$\\
	\end{tabular}
 \end{center} 
	 
\end{definition}

\begin{example}
Given the language $\mathcal{L}$ of the running Example \ref{ex:1}, a fully residuated functional $\mathcal{L}$-frame, $\mathbb{F}$, and an assignment $h:\mathrm{MVar}\to\mathbb{F}^+$ we have $$h\{\Delta \MLARR  \Gamma_3\}=\{w\mlarr^{\ast}u\mid w\in h(\Delta)^\uparrow\ \&\ u\in h(\Gamma_3)\}$$ and $$h\{\Gamma_1 \MAND \BDIA\Gamma_2\}=\{w\mand^\ast (\bdia^\ast u)\mid w\in h(\Gamma_1)\ \&\ u\in h(\Gamma_2)\}.$$    
\end{example}

\begin{lemma}\label{lemma:h curly is good}
    For any $h:\mathsf{MVar}\to \mathbb{F}^+$, any  $S\in \mathsf{MStr}_\mathcal{F}$  and $T\in \mathsf{MStr}_\mathcal{G}$, 
    \begin{align*}h\{S\}^{\uparrow\downarrow}=h(S)\text{ and }h\{T\}^{\downarrow}=h(T).\end{align*}
\end{lemma}
\begin{proof} Notice that $h\{T\}^{\downarrow}=h(T)$ is equivalent to $h\{T\}^{\downarrow\uparrow}=h(T)^\uparrow$.
    The proof proceeds by simultaneous induction on $S$ and $T$. The base case is immediate by Definition \ref{def:image}.1.  For the induction step, let $S$ be $\FH(\overline{S})$ for  $\overline{S}\in \mathsf{MStr}^{\epsilon_f}_\mathcal{F}$ and assume that the induction hypothesis holds for every $1\leq i\leq n_f$, i.e.\ $\gamma_N(h\{S_i^{1}\})=h(S_i^{1})$ if $\epsilon_{f,i}=1$ and $\gamma_N(h\{S_i^{\partial}\})= h(S_i^{\partial})^{\uparrow}$ if $\epsilon_{f,i}=\partial$. Then:
    \begin{align*}
	h\{\FH(\overline{S})\}^{\uparrow\downarrow} &=f^\ast[\overline{h\{S\}}]^{\uparrow\downarrow} & \mbox{Definition \ref{def:image}.2}\\
	&= f^\ast[\overline{\gamma_N(h\{S\})}]^{\uparrow\downarrow} & \mbox{Equation  \ref{eq:functional relation 3}}\\
    &= f^\ast[\overline{h(S)}]^{\uparrow\downarrow} & \mbox{induction hypothesis}\\
    &= R^{(0)}_f[\overline{h(S)}]^{\downarrow} & \mbox{Equation  \ref{eq:functional relation 2}}\\
	&= f_{R_f}(\overline{h(S)})  &  \mbox{Definition \ref{def:complexalg}}\\
	& = h(\FH(\overline{S})) &  \mbox{$h$ is a homomorphism,}
	\end{align*}
where the $i$-th coordinate of $\overline{h(S)}$ is $h(S_i^{1})$ if $\epsilon_{f,i}=1$ and $h(S_i^{\partial})^{\uparrow}$ if $\epsilon_{f,i}=\partial$.

The case in which $T$ is of the form $\GC(\overline{T})$ for  $\overline{T}\in \mathsf{MStr}^{\epsilon_g}_\mathcal{G}$ is shown similarly.
\end{proof}

\begin{lem}\label{prop:key soundness}
The following are equivalent:
\begin{enumerate}
	\item $h(S) \subseteq h(T)$;
\item	$s N t$ for every $s\in h\{S\} $ and $t \in h\{T\}$.
\end{enumerate}
\end{lem}
\begin{proof}
	$1\Rightarrow 2$. By Lemma \ref{lemma:h curly is good},  $h(S) \subseteq h(T)$ implies that $h\{S\}\subseteq h\{T\}^{\downarrow}$. This means that $sN t$ for every $s\in h\{S\} $ and $t \in h\{T\}$. 
	
	$2\Rightarrow 1$. If $s N t$ for every $s\in h\{S\} $ and $t \in h\{T\}$, then $h\{S\}\subseteq h\{T\}^{\downarrow}$, which by Lemma \ref{lemma:h curly is good} implies that $h\{S\}\subseteq h(T)$. Since $h(T)$ is a stable set, it follows that $h\{S\}^{\uparrow\downarrow}\subseteq h(T)$, which again by Lemma \ref{lemma:h curly is good} means that $h(S)\subseteq h(T)$, as required.
\end{proof}
\begin{remark}
    Lemmas \ref{lemma:h curly is good} and \ref{prop:key soundness} clarify the connection between the operators on the complex algebra of a fully residuated $\mathcal{L}$-frame and its functions defining the compatible relations. Indeed, given an assignment for variables on the complex algebra, Lemma \ref{lemma:h curly is good} says that to calculate the value of the composition of operators on the complex algebra it is enough to take the stable closure of the image of the valuation through the functions of the functional frame. Lemma \ref{prop:key soundness} shows then that term inequalities in the complex algebra can be verified through the images of the valuation through the functions of the frame. This property of fully residuated functional frames is key for showing the cut-elimination via a semantic route. As we show in the following section, a cut-free calculus can be semantically represented by a functional frame, since the  relation $\Rightarrow$ is not transitive, while the calculus augmented with cut can be semantically represented by the corresponding complex algebra of that frame. The tight connection between the frame and complex algebra illustrated by Lemma \ref{prop:key soundness} is key to show the semantic cut elimination for any analytic extension of the basic calculus via a Truth lemma style argument (see Lemma \ref{coro:lifting map} below) since it implies that the content of adding an analytic rule to a calculus with or without cut is the same, i.e.\ the algebraic quasi-equation corresponding to the rule. This is precisely formulated in Proposition \ref{prop:soundness} in the  following section.
\end{remark}

\section{Functional $\mathrm{D}$-frames and soundness of analytic structural rules on corresponding $\mathrm{D}$-frames}\label{sec:dframes}

In the present section, we introduce functional $\mathrm{D}$-frames, which are fully residuated functional frames arising from cut-free display calculi. 
Functional $\mathrm{D}$-frames are the counterparts, in the setting of LE-logics, of Gentzen frames \cite[Section 2]{galatos2013residuated}. We identify their key property in the form of a `Truth lemma', and use it to show that analytic structural rules are sound in their corresponding functional $\mathrm{D}$-frames.

\subsection{Functional $\mathrm{D}$-frames}
\label{ssec:functional D-frames}
Recall that  $\mathrm{D.LE}$ and $\mathrm{cfD.LE}$ respectively denote the display calculus for the basic normal $\mathcal{L}$-logic and
 its cut-free version. Moreover we let $\mathrm{D.LE'}$ and $\mathrm{cfD.LE'}$ denote the extensions of  $\mathrm{D.LE}$ and $\mathrm{cfD.LE}$ with some analytic structural rules.
\begin{definition}\label{def: dfram}
	  	 Let $\mathrm{D} \in \{\mathrm{D.LE}, \mathrm{D.LE'}, \mathrm{cfD.LE}, \mathrm{cfD.LE'}\}$. A  {\em functional $\mathrm{D}$-frame} is a structure $\mathbb{F}_{\mathrm{D}} \ceq (W, W^\partial, N, \mathcal{R}_{\mathcal{F}}, \mathcal{R}_{\mathcal{G}})$, where 
	\begin{enumerate}
                 \item $W \ceq \mathsf{Str}_\mathcal{F}$, $W^\partial \ceq \mathsf{Str}_\mathcal{G}$ and $N\subseteq W\times W^\partial$,
 		\item for every $f\in\mathcal{F}$ and $\overline{x} \in W^{\epsilon_f}$, $R_f(y, \overline{x})$ iff $\FH(\overline{x}) N y$,
		\item for every $g\in\mathcal{G}$ and $\overline{y} \in W^{\epsilon^\partial_g}$, $R_g(x,\overline{y})$ iff $ xN\GC(\overline{y})$ and 
		\item for any instance of any rule in $\mathrm{D}$ (including zero-ary rules)
		\[
		\AXC{$x_1\Rightarrow y_1, \ldots, x_n\Rightarrow y_n$}
		\UIC{$x \Rightarrow y$}
		\DP
		\]
		 $N$ is closed under the corresponding rule
		\[
		\AXC{$x_1 N y_1, \ldots, x_n N y_n$}
		\UIC{$x N y$}
		\DP
		\]
\end{enumerate}
\end{definition}
It is straightforward to show, by induction on the height of derivations in $\mathrm{D}$, that for every $x \in W$ and $y \in W^\partial$, if $\vdash_{\mathrm{D}} x \Rightarrow y$ then $x N y$. By definition, $\mathbb{F}_{\mathrm{D}}$ is functional (cf.~Definition \ref{def: LE frame}) for $f^\ast \ceq \FH$ and $g^\ast \ceq \GH$, and these maps have pseudo-residuals in each coordinate (namely, ${f^\ast}_i^\sharp \ceq \FHS_i$ and ${g^\ast}_i^\flat \ceq \GHF_i$ in each coordinate $i$), i.e.~$\mathbb{F}_{\mathrm{D}}$ is fully residuated (cf.~Section \ref{sec: functional L-frames}). 
As discussed in the same section, this implies that each associated relation $R_f$ and $R_g$ is $N$-compatible (cf.~Definition \ref{def: LE frame});  hence, 
$\mathbb{F}_{\mathrm{D}}$ is an $\mathcal{L}$-frame, and so, by Proposition  \ref{prop:F plus is L star algebra}, $\mathbb{F}_{\mathrm{D}}^+$ is a complete $\mathcal{L}$-algebra (and hence also an $\mathcal{L}^\ast$-algebra).

\subsection{`Truth lemma' in functional $\mathrm{D}$-frames}

In the present section, we let $\mathrm{D} \in \{\mathrm{D.LE}, \mathrm{D.LE'}, \mathrm{cfD.LE}, \mathrm{cfD.LE'}\}$. We will also denote $\overline{X}^{\partial}:= (X^{\partial}_1, \ldots, X_n^{\partial})$ and let $N^{\partial}$ be the converse of $N$.

\begin{lemma}\label{lemma:basiccalc}For any  functional $\mathrm{D}$-frame $\mathbb{F}$, any $f\in \mathcal{F}$  and $g\in \mathcal{G}$:
\begin{enumerate}
    \item If $\FH(\overline{x})Ny$  and each $x_i$ in $\overline{x}$ is a formula, then $f(\overline{x})Ny$. 
    \item If $\FH(\overline{x})\in X\subseteq W$ and each $x_i$ in $\overline{x}$ is a formula, then $f(\overline{x})\in\gamma_N(X)$.
    \item If $X_i\subseteq W^{\epsilon_{f,i}}$ and $\overline{x} \in \overline{X}^{\partial}$, and each $x_i$ in $\overline{x}$ is a formula, then $f(\overline{x})\in (\FH[\overline{X}])^\uparrow$.
    \item If $xN\GC(\overline{y})$ and each $y_i$ in $\overline{y}$  is a formula, then $xN g(\overline{y})$. 
    \item  If $\GC(\overline{y})\in Y\subseteq W^\partial$ and each $y_i$ in $\overline{y}$ is a formula, then $g(\overline{y})\in\gamma_N(Y)$.
    \item If $Y_i\subseteq (W^\partial)^{\epsilon_{g,i}}$ and $\overline{y} \in \overline{Y}^{\partial}$,  and each $y_i$ in $\overline{y}$  is a formula, then $g(\overline{y})\in (\GC[\overline{Y}])^\downarrow$.
\end{enumerate}
\end{lemma}
\begin{proof}
    1.  If $\FH(\overline{x}) N y$ and each $x_i$ is a formula,  by the rule  $(f_L)$ and Definition \ref{def: dfram}.4, we obtain that $f(\overline{x}) N y$. 

    2. Since $\FH(\overline{x})\in X$, $\FH(\overline{x})N y$ for all $y\in X^{\uparrow}$. By item 1, it follows that $f(\overline{x})N y$, for all $y\in X^{\uparrow}$. Hence $f(\overline{x})\in X^{\uparrow\downarrow}$.
    
    3. The assumption that $\overline{x} \in \overline{X}^{\partial}$ implies that $z_i N^{\epsilon_{f,i}}x_i$ for every $1 \leq i\leq n$ and every $z_i\in X_i$. By the rule  $(f_R)$ and Definition \ref{def: dfram}.4, we obtain that $\FH(\overline{z}) N f(\overline{x})$ for every $\overline{z}\in \overline{X}$, i.e.\ $f(\overline{x})\in (\FH[\overline{X}])^\uparrow$.
    The proofs of the remaining items are dual.
\end{proof}

\begin{lemma}\label{lem:join and meet}
For any  functional $\mathrm{D}$-frame $\mathbb{F}$, and all $\varphi,\psi \in \mathcal{L}$, 
	\begin{enumerate}
		\item $\bot N y$ for all $y\in W^\partial$ and $x N \top$ for all $x\in W$.
   
        \item If $Y_1, Y_2 \subseteq W^\partial$, $\varphi\in Y_1^{\downarrow}$ and $\psi\in Y_2^{\downarrow}$, then $\varphi\land\psi\in Y_1^{\downarrow}\cap Y_2^{\downarrow}$ and $\varphi\lor\psi\in (Y_1\cap Y_2)^{\downarrow}$.

	      \item If $X_1,X_2 \subseteq W$, $\varphi\in X_1^{\uparrow}$ and $\psi\in X_2^{\uparrow}$,  then $\varphi\lor\psi \in X_1^{\uparrow}\cap X_2^{\uparrow}$ and $\varphi\land\psi\in (X_1\cap X_2)^{\uparrow}$.
       
			\end{enumerate}
\end{lemma}
\begin{proof}
1. Rule $(\abot_L)$ implies that $\bot N \ABOT$, which implies  $\bot N y$ by rule $(\abot_W)$. The case for $\top$ is dual.

2. The assumptions $\varphi\in Y_1^{\downarrow}$ and $\psi\in Y_2^{\downarrow}$ are equivalent to $\varphi N y_1$ and $\psi N y_2$ for every $y_1 \in Y$ and $y_2 \in Y_2$. By the rule $(\mathrm{\wedge_L})$ and Definition \ref{def: dfram}.4, this implies that $(\varphi\land\psi) N y_1$ and $(\varphi\land\psi) N y_2$  for every $y_1 \in Y_1$ and $y_2 \in Y_2$, which shows that $\varphi\land\psi\in Y_1^{\downarrow}$ and $\varphi\land\psi\in Y_2^{\downarrow}$, i.e. $\varphi\land\psi\in Y_1^{\downarrow}\cap Y_2^{\downarrow}$, which proves the first part of the claim. 
As to the second part, the assumptions imply that $\varphi N y$ and $\psi N y$ for every $y \in Y_1\cap Y_2$ . By  the rule $(\mathrm{\vee_L})$ and Definition \ref{def: dfram}.4, we obtain that  $(\varphi\lor\psi) N y$ for every $y \in Y_1\cap Y_2$, therefore $\varphi\lor\psi\in (Y_1\cap Y_2)^{\downarrow}$ as required. The proof of item 3 is dual.
\end{proof}

\begin{lemma}[Truth lemma]\label{coro:lifting map}
For any  functional $\mathrm{D}$-frame $\mathbb{F}$, let $h: \mathsf{Str}_\mathcal{F} \cup \mathsf{Str}_\mathcal{G} \rightarrow \mathbb{F}^+_{\mathrm{D}}$ be the unique homomorphic extension of the assignment $p \mapsto \{p\}^\downarrow$.   Then $x \in h(x)$ for any $x \in \mathsf{Str}_\mathcal{F}$,  and  $y \in h(y)^{\uparrow}$ for any $y \in \mathsf{Str}_\mathcal{G}$. 
\end{lemma}
\begin{proof}
The proof proceeds by simultaneous induction on the number of structural and operational connectives of $x$ and $y$. If  $x$ or $y$ is an atomic proposition $p$, then $p\in p^{\downarrow} = h(p)$ because  the identity axiom in $\mathrm{D}$ implies that $pNp$; clearly, $p\in p^{\downarrow\uparrow} = h(p)^{\uparrow}$. If $x$ is the constant $\AATOP$ or $\top$, then $x\in W = h(\top)$. If $y$ is the constant $\top$, then by Lemma \ref{lem:join and meet}.1 $\top\in W^{\partial}=h(\top)^\uparrow$. The argument for  $y$ being $\ABOT$, $\bot$, or $x$ is $\bot$ is dual. 

If $x$ or $y$ is $\varphi \vee \psi$, 
 then  by induction hypothesis, $\varphi\in h(\varphi)^{\uparrow}$ and  $\psi\in h(\psi)^{\uparrow}$, so by Lemma \ref{lem:join and meet}.3, $ \varphi \vee \psi\in h(\varphi)^{\uparrow}\cap h(\psi)^{\uparrow} = h(\varphi\vee \psi)^\uparrow$. Moreover, by induction hypothesis, $\varphi\in h(\varphi) =h(\varphi)^{\uparrow\downarrow}$ and  $\psi\in h(\psi) = h(\psi)^{\uparrow\downarrow}$, so by Lemma \ref{lem:join and meet}.2, $ \varphi \vee \psi\in (h(\varphi)^{\uparrow}\cap h(\psi)^{\uparrow})^\downarrow = h(\varphi \vee \psi)$.
  The proof for  $x$ or $y$ being $\varphi \wedge \psi$ is dual. 
  
  If $x$ is $\FH(\overline{x'})$, then by induction hypothesis $\overline{x'}\in\overline{h(x')}^{\epsilon_{f}}$, hence
    $$\FH(\overline{x'})\in\FH[\overline{h(x')}]\subseteq \FH[\overline{h(x')}]^{\uparrow\downarrow}= R^{(0)}_{f}[\overline{h(x')}^{\epsilon_{f}}]^{\downarrow}= h(x).$$ If $x$ is $f(\overline{\psi})$, then by what was shown above $\FH(\overline{\psi})\in h(\FH(\overline{\psi}))=h(f(\overline{\psi}))$, which implies, by Lemma \ref{lemma:basiccalc}.2, that $f(\overline{\psi})\in \gamma_N(h(f(\overline{\psi}))) = h(f(\overline{\psi}))$. If $y$ is $f(\overline{\psi})$, by induction hypothesis $\overline{\psi}\in\overline{h(\psi)}^{\epsilon_f^\partial}$. Then, by Lemma \ref{lemma:basiccalc}.3, $y\in\FH[\overline{h(\psi)^{\epsilon_f}}]^{\uparrow}=h(y)^{\uparrow}$. The cases in which $y$ is $\GC(\overline{y})$, $g(\overline{\psi})$ and $x$ is $g(\overline{y})$ are dual. 
\end{proof}

\begin{prop}\label{prop:main cut elimination}
For every sequent $x \Rightarrow y$, if $\mathbb{F}_\mathrm{D}^+ \models x \Rightarrow y$ then $ x N y$ in $\mathbb{F}_\mathrm{D}$.
\end{prop} 
\begin{proof}
Let $h: \mathsf{Str}_\mathcal{F} \cup \mathsf{Str}_\mathcal{G} \rightarrow \mathbb{F}^+_{\mathrm{D}}$ be the unique homomorphic extension of the assignment $p \mapsto \{p\}^\downarrow$. If $\mathbb{F}^+_{\mathrm{D}} \models x \Rightarrow y$ then $h(x) \subseteq h(y)$. By Lemma \ref{coro:lifting map}, $x \in h(x)$ and $y \in h(y)^\uparrow$. From $x \in h(x)$ and $h(x) \subseteq h(y)$ it follows $x\in h(y)$, which implies, since $y\in h(y)^\uparrow$,  that $xNy$. 
\end{proof}

\subsection{Soundness of analytic structural rules in the complex algebras of functional $\mathrm{D}$-frames}

 In  the present subsection, we let $\mathrm{D} \in \{\mathrm{D.LE}, \mathrm{D.LE'}, \mathrm{cfD.LE}, \mathrm{cfD.LE'}\}$ and show that if $\mathrm{D}$ is obtained by extending the basic calculus $\mathrm{D.LE}$ with analytic structural rules, then these additional rules are sound in the complex algebras of any functional $\mathrm{D}$-frame (cf.~Proposition \ref{prop:soundness}). From this, it immediately follows that the analytic inductive inequalities from which these rules arise are valid in these algebras. 

\begin{prop}\label{prop:soundness}
Let $\mathrm{D} \in \{\mathrm{D.LE}, \mathrm{D.LE'}, \mathrm{cfD.LE}, \mathrm{cfD.LE'}\}$. The rules of $\mathrm{D}$ are sound in $\mathbb{F}^+_{\mathrm{D}}$.
\end{prop}
\begin{proof}
As discussed in Section \ref{ssec:functional D-frames}, $\mathbb{F}_{\mathrm{D}}$ is an $\mathcal{L}$-frame, and hence, by Proposition \ref{prop:F plus is L star algebra}, $\mathbb{F}^+_{\mathrm{D}}$ is a complete $\mathcal{L}$-algebra, and hence an $\mathcal{L}^*$-algebra. Therefore, all the rules which $\mathrm{D}$ shares with $\mathrm{cfD.LE}$ are sound. If $\mathrm{R}$ is an analytic structural rule of $\mathrm{D}$, then $\mathrm{R}$ has the following shape:
\[
\AXC{$S_1 \fCenter T_1$}
\AXC{$\cdots$}
\AXC{$S_n \fCenter T_n$}
\RL{\small $\mathrm{R}$}
\TIC{$S_0 \fCenter T_0$}
\DP
\]
with the assumptions on the meta-variables reported on at the end of  Section \ref{ssec:syntactic frames associated w algebras}. By Definition \ref{def: dfram}.4, 
\[
\AXC{$s_1 N t_1$}
\AXC{$\cdots$}
\AXC{$s_n N t_n$}
\RightLabel{\small $\mathrm{R_N}$}
\TIC{$s_0 N t_0$}
\DP
\]
holds in $\mathbb{F}_{\mathrm{D}}$ where $s_i$ (resp.~$t_i$) is an instantiation of $S_i$ (resp.~$T_i$) with $1 \leq i \leq n$. 

Let $h:\mathsf{MVar}\to\mathbb{F}^+_{\mathrm{D}}$ be an assignment of meta-variables, which we identify with its unique homomorphic extension, and assume that $h(S_1) \subseteq h(T_1), \ldots, h(S_n) \subseteq h(T_n)$. We need to prove that $h(S_0) \subseteq h(T_0)$.  
By Lemma \ref{prop:key soundness},  this is equivalent to showing that $s_0 N t_0$ for every $s_0 \in h\{S_0\} $ and $t_0 \in h\{T_0\}$. Notice that since the rule $\mathrm{R}$ is analytic, each meta-variable in $s_0 N t_0$ occurs at most once, and therefore $$h\{S_0(\Gamma_1, \ldots, \Gamma_k)\} = \{s_0(x_1, \ldots, x_k)\mid x_j\in h\{\Gamma_j\} \mbox{ for } 1\leq j\leq k\}$$ and $$h\{T_0 (\Delta_1, \ldots, \Delta_m)\} = \{t_0(y_1, \ldots, y_m)\mid y_j\in h\{\Delta_j\} \mbox{ for } 1\leq j\leq m\}.$$ Hence,  each sequent  $s_0 N t_0$ is an instance of  the conclusion of $\mathrm{R_N}$ and induces a choice function $\eta: \mathsf{MVar}\to \bigcup h[\mathsf{MVar}]$ such that $\eta(\Gamma)\in h\{\Gamma\}$. We then let   $\{s_i N t_i \mid 1 \leq i \leq n\}$ be the corresponding instance of  the premises of $\mathrm{R_N}$, in the sense that e.g.~$s_i = s_i(\eta(\Gamma_1), \ldots \eta(\Gamma_k))$ for each $i$. 
Since $\mathrm{R_N}$ holds in $\mathbb{F}_D$, to prove our claim it is enough to show that  $s_iN t_i$ holds in $\mathbb{F}_{\mathrm{D}}$ for each $i$. This is guaranteed by the assumption $h(S_i) \subseteq h(T_i)$ and by Lemma \ref{prop:key soundness}, since, by Definition \ref{def:image},  $s_i\in h\{S_i\}$ and $t_i\in h\{T_i\}$.
\end{proof}

Let $\mathbf{L}'$ be the analytic extension of   the basic $\mathcal{L}$-logic $\mathbf{L}$ with the analytic inductive axioms  corresponding to the additional rules of $\mathrm{D}$.   The proposition above immediately implies the second part of the following statement, while the first part has been discussed right after Definition \ref{def: dfram}. 
 
 \begin{thm}\label{cor:complex algebra is LE'-algebra}
 If $\mathbb{F}_{\mathrm{D}}$ is a functional $\mathrm{D}$-frame for  $\mathrm{D} \in \{\mathrm{D.LE}, \mathrm{cfD.LE}\}$, then $\mathbb{F}_{\mathrm{D}}^+$ is a complete $\mathrm{L}$-algebra. 
 If $\mathbb{F}_{\mathrm{D}}$ is a functional $\mathrm{D}$-frame for $\mathrm{D} \in \{\mathrm{D.LE'}, \mathrm{cfD.LE'}\}$,  then $\mathbb{F}_{\mathrm{D}}^+$ is a complete $\mathbf{L}'$-algebra.
\end{thm}

The Proposition and Theorem above show that if a cut-free display calculus $D$ contains a structural rule then the complex algebra of its $D$-frame will satisfy the corresponding quasi-inequality, hence be in the class of algebras w.r.t.\ which the corresponding calculus with cut is sound and complete. Key to this result are the properties of functional frames discussed in Lemmas  \ref{lemma:h curly is good} and \ref{prop:key soundness}. This allows for a uniform proof of cut-elimination, which is presented in the following section, for all axiomatic extensions of $\mathrm{LE}$-logics whose axioms correspond to analytic rules.

\section{Semantic cut elimination}\label{sec:semcut}
 In the present section, we  fix an arbitrary LE-signature  $\mathcal{L} = \mathcal{L}(\mathcal{F}, \mathcal{G})$  and show the semantic cut elimination  for the proper display calculi associated with the basic normal $\mathcal{L}$-logic   $\mathbf{L}$ and all its analytic extensions $\mathbf{L'}$.\footnote{By Proposition \ref{prop:type5}, analytic extensions of $\mathbf{L}$ are captured by proper display calculi $\mathrm{D.LE'}$  obtained by adding analytic structural rules  to the basic calculus  $\mathrm{D.LE}$. By the general theory developed in \cite{greco2018unified}, $\mathrm{D.LE'}$ is sound and complete with respect to the class of complete $\mathcal{L}$-algebras validating the additional axioms.}  Let $\mathrm{D} \in \{\mathrm{D.LE}, \mathrm{D.LE'}\}$,  and let $\mathrm{cfD}$ be the cut-free version of $\mathrm{D}$.
 
 \begin{thm}\label{thm:semantic cut elim ax ext}
For every sequent $x \Rightarrow y$, if $ \vdash_{\mathrm{D}} x \Rightarrow y$ then $\vdash_{\mathrm{cfD} } x \Rightarrow y$.
\end{thm}
\begin{proof}
Let $\mathbb{F}_{\mathrm{cfD}}$ be the functional $\mathrm{cfD}$-frame (cf.~Definition \ref{def: dfram}) in which $N$ is defined as follows: for all $x \in W$ and $y \in W^\partial$,  \begin{equation}\label{def:eN}
x N y \quad\mbox{ iff } \quad\vdash_{\mathrm{cfD}} x \Rightarrow y.
\end{equation}

 The proof strategy is  illustrated by the following diagram. 
 
\[
\begin{tikzpicture}[node/.style={circle, draw, fill=black}, scale=1]
\node (a_1) at (-1.5, 0) {$\mathbb{F}^+_{\mathrm{cfD}}\models x \Rightarrow y$};
\node (b_1) at (2.5, 0) {$x N y$ in $\mathbb{F}_\mathrm{cfD}$};
\node (a_2) at (-1.5, 1.5) {$\vdash_\mathrm{D} x \Rightarrow y$};
\node (b_2) at (2.5, 1.5) {$\vdash_\mathrm{cfD} x \Rightarrow y$};
\draw [dashed, ->] (a_2)  to (b_2);
\draw [->] (a_1)  to (b_1);
\draw [->] (a_2)  to (a_1);
\draw [<->] (b_1)  to (b_2);
\end{tikzpicture}
\]
By Theorem \ref{cor:complex algebra is LE'-algebra}, 
$\mathbb{F}_\mathrm{cfD}^+$ is a complete $\mathbf{L}$-algebra (resp.~$\mathbf{L'}$-algebra). If $\mathrm{D} = \mathrm{D.LE}$, then by Theorem \ref{prop:soundness of D.LE}, $\mathrm{D}$ is sound w.r.t.~the class of complete $\mathbf{L}$-algebras;  if $\mathrm{D} = \mathrm{D.LE'}$ for some analytic extension $\mathbf{L'}$ of  $\mathbf{L}$, then  by the general theory developed in \cite{greco2018unified},   $\mathrm{D}$ is sound  w.r.t.~the class of complete $\mathbf{L'}$-algebras. 
In either case,  $\vdash_\mathrm{D} x \Rightarrow y$  implies that $\mathbb{F}_\mathrm{cfD}^+ \models x \Rightarrow y$, which is the vertical arrow on the left-hand side of the diagram. The horizontal implication follows from Proposition \ref{prop:main cut elimination}.
The vertical equivalence on the right-hand side of the diagram holds by  \eqref{def:eN}.
\end{proof}

 \section{Finite model property}\label{sec:fmp}
 In this section we discuss a general strategy for obtaining the finite model property and hence decidability of extensions of $\mathrm{LE}$-logics with rules of particular shape. The finite model property is obtained via an appropriate modification of functional $\mathrm{D}$-frames.
 
We say that a display calculus has the {\em finite model property} (FMP) if every sequent $x \Rightarrow y$ that is not derivable in  the calculus has a finite counter-model. In this section, we prove the FMP for  $\mathrm{D}\in \{\mathrm{D.LE}, \mathrm{D.LE'}\}$ where $\mathrm{D.LE}$ is the display calculus  for the basic $\mathrm{LE}$-logic, and  $\mathrm{D.LE'}$ is one of its extensions with analytic structural rules subject to certain conditions (see below). For any sequent $x \Rightarrow y$ such that $\nvdash_{\mathrm{D}} x\Rightarrow y$, our proof strategy consists in constructing a functional $\mathrm{D}$-frame $\mathbb{F}^{x\Rightarrow y}_\mathrm{D}$  the complex algebra of which is finite. The basic idea to  satisfy the requirement of finiteness is provided by the following lemma (the symbol $(\cdot)^c$ denotes the relative complementation). 

\begin{lem}\label{lem: main idea FMP}
	Let $\mathbb{W}=(W,W^\partial,N)$ be a polarity. If  the set  $\{ y^\downarrow\mid y\in W^\partial\}$ is finite, then $\mathbb{W}^+$ is finite. Dually, if the set  $\{ x^\uparrow\mid x\in W\}$ is finite, then $\mathbb{W}^+$ is finite.
\end{lem}
 \begin{proof}
 Since $\{ y^\downarrow\mid y\in W^\partial\}$ meet-generates  $\mathbb{W}^+$, an upper bound to the size of $\mathbb{W}^+$ is  $2^{|\{ y^\downarrow\mid y\in W^\partial\}|}$. The remaining part of the statement is proven dually.
 \end{proof}

 \begin{definition}\label{def:proof search}
 Let $\mathrm{D}$ be a display calculus as above. For any sequent $x \Rightarrow y$, let $(x \Rightarrow y)^\leftarrow$ be the set of sequents which is defined recursively as follows:
 \begin{enumerate}
 \item $x \Rightarrow y \in (x \Rightarrow y)^\leftarrow$;
 \item if {\small\AXC{$x_1 \Rightarrow y_1, \ldots, x_n \Rightarrow y_n$}
 \UIC{$x_0 \Rightarrow y_0$}
 \DP}
is an instance of a rule in $\mathrm{D}$, and $x_0 \Rightarrow y_0 \in (x \Rightarrow y)^\leftarrow$, then $x_1 \Rightarrow y_1, \ldots, x_n \Rightarrow y_n \in (x \Rightarrow y)^\leftarrow$.
\end{enumerate}
\end{definition}

\begin{definition}\label{def:countermodel} For any sequent $x\Rightarrow y$, 
	let $\mathbb{F}^{x\Rightarrow y}_{\mathrm{D}}$ denote  the structure $(W, W^\partial, N, \mathcal{R}_{\mathcal{F}}, \mathcal{R}_{\mathcal{G}})$ such that $W, W^\partial,  \mathcal{R}_{\mathcal{F}}, \mathcal{R}_{\mathcal{G}}$ are as in Definition \ref{def: dfram}.1-3, and $N$ is defined as follows:
	\begin{equation}\label{eq:FMP}
w N u \quad\mbox{ iff }\quad  \vdash_{\mathrm{D}} w \Rightarrow u  \mbox{ or } w \Rightarrow u \not\in (x \Rightarrow y)^\leftarrow.
	\end{equation}
\end{definition}

 \begin{prop}
$\mathbb{F}^{x\Rightarrow y}_{\mathrm{D}}$ is a functional $\mathrm{D}$-frame.
 \end{prop}
\begin{proof}
We only need to show that $\mathbb{F}^{x\Rightarrow y}_{\mathrm{D}}$ satisfies Definition \ref{def: dfram}.4, i.e.~for every rule $\mathrm{R}$:
\[\AXC{$x_1 \Rightarrow y_1, \ldots, x_n \Rightarrow y_n$}
 \UIC{$x_0 \Rightarrow y_0$} 
 \DP\]
 in $\mathrm{D}$, 
 \[\AXC{$x_1 N y_1, \ldots, x_n N y_n$}
 \UIC{$x_0 N y_0$} 
 \DP\]
 holds. Assume that $x_1 N y_1, \ldots, x_n N y_n$. If  $\vdash_{\mathrm{D}} x_1 \Rightarrow y_1, \ldots, \vdash_{\mathrm{D}} x_n \Rightarrow y_n$, then $\vdash_{\mathrm{D}} x_0 \Rightarrow y_0$ by applying $\mathrm{R}$, hence $x_0 N y_0$ by the definition of $N$. Otherwise,  $\not\vdash_{\mathrm{D}}x_i \Rightarrow y_i$ for some  $1 \leq i \leq n$, and hence \eqref{eq:FMP}
 and the assumption $x_iNy_i$ imply that $x_i \Rightarrow y_i \not\in (x \Rightarrow y)^\leftarrow$. Hence, $x_0 \Rightarrow y_0 \not\in (x \Rightarrow y)^\leftarrow$ by Definition \ref{def:proof search}.2. Therefore, we conclude again that $x_0 N y_0$.
\end{proof}
The above proposition and Theorem \ref{cor:complex algebra is LE'-algebra} imply that the complex algebra of $\mathbb{F}^{x\Rightarrow y}_{\mathrm{D}}$ is a complete $\mathcal{L}$-algebra if $\mathrm{D}$ is $\mathrm{D.LE}$ (resp.~a complete $\mathbb{L}_{\mathrm{LE'}}$-algebra if $\mathrm{D}$ is $\mathrm{D.LE'}$).

 \begin{prop}\label{prop:counter model}
	If $\not\vdash_{\mathrm{D}}  x \Rightarrow y$, then $(\mathbb{F}^{x\Rightarrow y}_{\mathrm{D}})^+ \not\vDash x \Rightarrow y$.
\end{prop}

\begin{proof}
	Let $h: \mathsf{Str}_\mathcal{F} \cup \mathsf{Str}_\mathcal{G} \rightarrow (\mathbb{F}^{x\Rightarrow y}_{\mathrm{D}})^+$ be the unique homomorphic extension of the assignment $p \mapsto \{p\}^\downarrow$. We will show that $h(x) \not\subseteq h(y)$. Assume that $h(x) \subseteq h(y)$. By Lemma \ref{coro:lifting map}, we obtain that $x \in h(x)$ and $h(y) \subseteq y^\downarrow$. Hence $x \in y^\downarrow$, i.e.~$x N y$, that is $\vdash_{\mathrm{D}} x \Rightarrow y$ or $x \Rightarrow y \notin (x \Rightarrow y)^\leftarrow$ by \eqref{eq:FMP}.  Since $x \Rightarrow y \in (x \Rightarrow y)^\leftarrow$ by Definition \ref{def:proof search}.1, we obtain that $\vdash_{\mathrm{D}} x \Rightarrow y$, which  contradicts $ \not\vdash_{\mathrm{D}} x \Rightarrow y$, and hence $h(x) \not\subseteq h(y)$, i.e. $(\mathbb{F}^{x\Rightarrow y}_{\mathrm{D}})^+ \not\vDash x \Rightarrow y$.
\end{proof}

Thus, the algebra $(\mathbb{F}^{x\Rightarrow y}_{\mathrm{D}})^+$ is a good candidate for the finite model property, provided we can define conditions under which it is finite.

\begin{definition}\label{def:provequiv}
Let $\varphi_{\mathcal{F}}$ denote the following equivalence relation on $\mathsf{Str}_{\mathcal{F}}$:  if $x$ and $x'$ are $\mathcal{F}$-structures,
$(x, x')\in \varphi_{\mathcal{F}}$ iff the following rule scheme is derivable in $\mathrm{D}$:
\begin{center}
\AX$x \fCenter y$
\doubleLine
\UI$x' \fCenter y$
\DP
\end{center}
An equivalence relation $\varphi_{\mathcal{G}}$  on $\mathsf{Str}_{\mathcal{G}}$ can be defined analogously. In what follows, we will let $[x']_{\varphi_{\mathcal{F}}}$  and $[y']_{\varphi_{\mathcal{G}}}$  denote the equivalence classes induced by $\varphi_{\mathcal{F}}$  and $\varphi_{\mathcal{G}}$ respectively.
\end{definition}

\begin{definition}
For every sequent $x\Rightarrow y$, let 
\[(x\Rightarrow y)^{\leftarrow}_{\mathcal{F}} \ceq \{x'\in \mathsf{Str}_{\mathcal{F}}\mid x'\Rightarrow y'\in (x\Rightarrow y)^{\leftarrow}\mbox{ for some }y'\in \mathsf{Str}_{\mathcal{G}}  \}\]
\[(x\Rightarrow y)^{\leftarrow}_{\mathcal{G}} \ceq \{y'\in \mathsf{Str}_{\mathcal{G}}\mid x'\Rightarrow y'\in (x\Rightarrow y)^{\leftarrow}\mbox{ for some }x'\in \mathsf{Str}_{\mathcal{F}}  \}.\]
\end{definition}

In what follows, we let $y^\downarrow:= \{x \in W \mid x N y\}$, where $N$ is defined as in \eqref{eq:FMP}. 
\begin{prop}\label{prop:corecountermodel}
For all $y' \in \mathsf{Str}_\mathcal{G}$ and $x' \in \mathsf{Str}_\mathcal{F}$ such that $(y'^\downarrow)^c\neq \varnothing$ and $(x'^\uparrow)^c\neq \varnothing$,
$$(y'^\downarrow)^c = \bigcup \{[x'']_{\varphi_\mathcal{F}} \mid x'' \in A\} \quad \mbox{ and } \quad (x'^\uparrow)^c = \bigcup \{[y'']_{\varphi_\mathcal{G}} \mid y'' \in B\} $$
for some $A \subseteq (x\Rightarrow y)^{\leftarrow}_{\mathcal{F}}$ and $B \subseteq (x\Rightarrow y)^{\leftarrow}_{\mathcal{G}}$.
\end{prop}
\begin{proof}
	Let $y'\in\mathsf{Str}_\mathcal{G}$. If $y'\notin (x\Rightarrow y)^{\leftarrow}_{\mathcal{G}}$ then $w\Rightarrow y'\notin(x\Rightarrow y)^{\leftarrow}$ for all $w\in \mathsf{Str}_\mathcal{F}$ and therefore,  by Definition \ref{def:countermodel},  $y'^{\downarrow}=\mathsf{Str}_\mathcal{F}$, i.e.\ $(y'^\downarrow)^c = \varnothing$. Therefore, we can assume without loss of generality that $y'\in  (x\Rightarrow y)^{\leftarrow}_{\mathcal{G}}$. Let $(x',x'')\in\varphi_{\mathcal{F}}$. Definition \ref{def:provequiv} implies that for every $u\in\mathsf{Str}_\mathcal{G}$ \begin{equation}
	\label{eq:firstpointofequiv}\vdash_{\mathrm{D}} x'\Rightarrow u\text{ if and only if }\vdash_{\mathrm{D}} x''\Rightarrow u
	\end{equation}  and  \begin{equation}
	\label{eq:secondpointofequiv} x'\Rightarrow u\in (x\Rightarrow y)^{\leftarrow}\text{ if and only if }x''\Rightarrow u\in (x\Rightarrow y)^{\leftarrow}.
	\end{equation} By Definition \ref{def:countermodel},  \eqref{eq:firstpointofequiv} and \eqref{eq:secondpointofequiv} we obtain \begin{equation}
	\label{eq:importanceofequiv}x'N y'\text{ if and only if }x''Ny'
	\end{equation} for every $(x',x'')\in\varphi_{\mathcal{F}}$. Furthermore, by Definition \ref{def:countermodel}, $w \cancel{N}y'$ implies that $w\Rightarrow y'\in (x\Rightarrow y)^{\leftarrow}$ and therefore $w\in(x\Rightarrow y)^{\leftarrow}_{\mathcal{F}}$. This combined with \eqref{eq:importanceofequiv} implies that there exists some $A\subseteq(x\Rightarrow y)^{\leftarrow}_{\mathcal{F}}$ such that $(y'^\downarrow)^c = \bigcup \{[x'']_{\varphi_\mathcal{F}} \mid x'' \in A\}$. The proof for $(x'^{\uparrow})^{c}$ is shown dually.     
\end{proof}

\begin{cor}\label{cor:fmpgeneral}
For every sequent $x\Rightarrow y$, 
\begin{enumerate}
\item if $\{[x']_{\varphi_{\mathcal{F}}}\mid x'\in (x\Rightarrow y)^{\leftarrow}_{\mathcal{F}}\}$ is finite, then $(\mathbb{F}^{x\Rightarrow y}_{\mathrm{D}})^+$ is finite.
\item if $\{[y']_{\varphi_{\mathcal{G}}}\mid y'\in (x\Rightarrow y)^{\leftarrow}_{\mathcal{G}}\}$ is finite, then $(\mathbb{F}^{x\Rightarrow y}_{\mathrm{D}})^+$ is finite.
\end{enumerate}
\end{cor}
\begin{proof}
By Proposition \ref{prop:corecountermodel}, for every $y'\in \mathsf{Str}_\mathcal{G}$, $(y'^{\downarrow})^{c}=\varnothing$, or  $(y'^\downarrow)^c = \bigcup \{[x'']_{\varphi_\mathcal{F}} \mid x'' \in A\}$ for some $A\subseteq (x\Rightarrow y)^{\leftarrow}_{\mathcal{F}}$. If $\{[x']_{\varphi_{\mathcal{F}}}\mid x'\in (x\Rightarrow y)^{\leftarrow}_{\mathcal{F}}\}$ is finite, then $\{(y'^{\downarrow})^{c}\mid y'\in \mathsf{Str}_\mathcal{G}\}$ is finite, so $\{y'^{\downarrow}\mid y'\in \mathsf{Str}_\mathcal{G}\}$ is finite, therefore Lemma \ref{lem: main idea FMP} implies that $(\mathbb{F}^{x\Rightarrow y}_{\mathrm{D}})^+$ is finite. Item 2 is shown analogously. 
\end{proof}

Proposition \ref{prop:counter model} and Corollary \ref{cor:fmpgeneral} imply the following:

\begin{thm}\label{thm:FMP}
	If the calculus $\mathrm{D}$ verifies one of the assumptions of Corollary \ref{cor:fmpgeneral} then FMP holds for $\mathrm{D}$. 
\end{thm}

In what follows we will discuss sufficient conditions for the assumptions of Corollary \ref{cor:fmpgeneral} to hold.

\begin{prop}\label{prop:compdecrfmp}
If all rules in $\mathrm{D}$ applied bottom up  decrease or leave unchanged the complexity of sequents, then FMP holds for $\mathrm{D}$. 
\end{prop} 
\begin{proof}
	The assumptions imply that the set $(x \Rightarrow y)^\leftarrow$ is finite and therefore the assumptions of Corollary \ref{cor:fmpgeneral} are satisfied.
\end{proof}
	
\begin{prop}\label{prop:refinedfmp}\begin{enumerate}
		\item If $\varphi'_{\mathcal{F}}$ is an equivalence relation such that $\varphi'_{\mathcal{F}}\subseteq\varphi_{\mathcal{F}}$ and moreover $\{[x']_{\varphi'_{\mathcal{F}}}\mid x'\in (x\Rightarrow y)^{\leftarrow}_{\mathcal{F}}\}$ is finite, then the FMP holds for $\mathrm{D}$.  
		\item If $\varphi'_{\mathcal{G}}$ is an equivalence relation such that  $\varphi'_{\mathcal{G}}\subseteq\varphi_{\mathcal{G}}$ and moreover $\{[x']_{\varphi'_{\mathcal{G}}}\mid x'\in (x\Rightarrow y)^{\leftarrow}_{\mathcal{F}}\}$ is finite, then the FMP holds for $\mathrm{D}$. 
	\end{enumerate}
\end{prop}
\begin{proof}
1. If $\varphi'_{\mathcal{F}}\subseteq\varphi_{\mathcal{F}}$, then every equivalence class of $\varphi_{\mathcal{F}}$ is the union of equivalence classes of $\varphi'_{\mathcal{F}}$. Hence, the assumption that $\{[x']_{\varphi'_{\mathcal{F}}}\mid x'\in (x\Rightarrow y)^{\leftarrow}_{\mathcal{F}}\}$ is finite guarantees that the assumptions of Corollary \ref{cor:fmpgeneral} are satisfied, and hence the statement follows by Theorem \ref{thm:FMP}.
\end{proof}
 
The proposition above provides us with an effective strategy to prove the FMP. Indeed, while computing $\varphi_{\mathcal{F}}$ or $\varphi_{\mathcal{G}}$ can be practically unfeasible, it is in fact enough to produce a suitable refinement of them. We will illustrate this technique in Sections \ref{ssec:ortho}. 

 



\section{Examples}\label{sec:exfmp}
 
In the present section, we will obtain cut elimination and FMP for concrete instances of LE-logics as a consequence of Theorems \ref{thm:semantic cut elim ax ext} and \ref{thm:FMP}.

\subsection{Basic epistemic logic of categories and running example}

The language of the basic {\em epistemic logic of categories} \cite{conradie2016categories,CFPPTW17}, denoted $\mathcal{L}_{\mathrm{ML}}$, is obtained by instantiating $\mathcal{F}: = \varnothing$ and $\mathcal{G} = \{\wbox\}$ with $n_\wbox = 1$ and $\varepsilon_\wbox = 1$. 

Clearly, Theorem \ref{thm:semantic cut elim ax ext} applies to the calculus $\mathrm{D.LE}$ for the basic $\mathcal{L}_{\mathrm{ML}}$-logic and to any calculus $\mathrm{D.LE}'$ obtained by adding any analytic structural rule to $\mathrm{D.LE}$ for instance those corresponding to the factivity and positive introspection axioms, $\wbox p\leq p$ and $\wbox\wbox p\leq \wbox p$ respectively. Moreover,  Proposition \ref{prop:compdecrfmp} applies to $\mathrm{D.LE}$ and any calculus $\mathrm{D.LE}'$ obtained by adding any analytic structural rule to $\mathrm{D.LE}$ such that the complexity of sequents does not increase from bottom to top.
This result covers FMP for  the display calculi capturing the epistemic logic of categories with positive introspection, since the corresponding structural rule is
\begin{center}
\AX$\Gamma\fCenter\WBOX\Delta$
\UI$\Gamma\fCenter\WBOX\WBOX\Delta$
\DP
\end{center}
which satisfies the conditions of Proposition \ref{prop:compdecrfmp}, since the size of the sequents decreases when the rule is applied bottom up.

Theorem \ref{thm:semantic cut elim ax ext} also applies for the logic of Example \ref{ex:1} and all its analytic extensions, e.g.\ the extension by the rule given in Example \ref{ex:rule}. Furthermore, that rule satisfies the conditions of Proposition \ref{prop:compdecrfmp}, since the size of the sequent is unchanged when the rule is applied bottom up, hence the extension of Example \ref{ex:rule} also has the FMP.

\subsection {Full Lambek calculus}

The language of the {\em full Lambek calculus} \cite{galatos2007residuated}, denoted $\mathcal{L}_{\mathrm{FL}}$, is obtained by instantiating $\mathcal{F} = \{e, \circ\}$ with $n_e = 0$, $n_\circ  = 2$, $\varepsilon_\circ = (1, 1)$   and   $\mathcal{G} = \{\backslash, /\}$ with $n_\backslash = n_/ = 2$, $\varepsilon_\backslash = (\partial, 1)$  and $\varepsilon_/ = (1, \partial)$. 

Clearly, Theorem \ref{thm:semantic cut elim ax ext} applies to the calculus $\mathrm{D.LE}$ for the basic $\mathcal{L}_{\mathrm{FL}}$-logic and to any calculus $\mathrm{D.LE}'$ obtained by adding any analytic structural rule to $\mathrm{D.LE}$. This result covers the semantic cut elimination for any display calculus for axiomatic extensions of the basic $\mathcal{L}_{\mathrm{FL}}$-logic with $\mathcal{N}_2$ axioms (cf.~\cite{ciabattoni2012algebraic}).  Moreover,  Proposition \ref{prop:compdecrfmp} applies to $\mathrm{D.LE}$ and any calculus $\mathrm{D.LE}'$ obtained by adding any analytic structural rule to $\mathrm{D.LE}$ such that the complexity of sequents does not increase from bottom to top.
This result covers FMP for  the display calculi capturing the nonassociative full Lambek calculus (cf.\ \cite{buszkowski2009nonassociative}),  the full Lambek calculus (which corresponds to $\mathrm{D.LE}$ plus associativity),  and its axiomatic extensions with commutativity, weakening, and simple rules that do not increase the complexity of sequents from bottom to top (cf.~\cite[Theorem 3.15]{galatos2013residuated}).

\subsection{Full Lambek-Grishin calculus}
\label{sec:FullLambekGrishincalculus}
The language of the {\em full Lambek-Grishin calculus} (cf.\ \cite{moortgat2007symmetries}), denoted $\mathcal{L}_{\mathrm{FLG}}$, is given by $\mathcal{F} = \{\circ, \starfor, \starback \}$ and $\mathcal{G} = \{\star, \circfor, \circback\}$ such that $\varepsilon_\circ = \varepsilon_\star = (1, 1)$,  $\varepsilon_{\starback} = \varepsilon_{\circback} = (\partial, 1)$ and $\varepsilon_{\starfor} = \varepsilon_{\circfor} = (1, \partial)$. 

One can explore the space of the axiomatic extensions of the basic $\mathcal{L}_{\mathrm{FLG}}$-logic with the following {\em Grishin interaction principles} \cite{grishin1983}:

\begin{center}
\begin{tabular}{c l @{} c @{} l @{} c c l @{} c @{} l  c c }
(a)&$(p\star q)\circ r    $&\ $\vdash$\ \,&$ p\star(q\circ r)       $&$\quad$& 
(d)&$(p\starback q)\circback r $&\ $\vdash$\ \,& $q\circback (p\star r)$&$\quad$&\\
(b)&$ p\star(q\circfor r) $&\ $\vdash$\ \,&$(p\star q)\circfor r    $&$\quad$& 
(e)&$(p\circ q)\starfor r $&\ $\vdash$\ \,&$ p\starfor (r\circfor q)$&$\quad$&(I)\\
(c)&$ p\starback(q\circ r)$&\ $\vdash$\ \,&$(p\starback q)\circ r   $&$\quad$& 
(f)&$ p\circ(q\circback r)$&\ $\vdash$\ \,&$(p\starfor q)\star r    $&$\quad$&\\
   & &      & &           & &      & &       &\\
(a)&$(p\circback q)\circ r   $&\ $\vdash$\ \,&$ p\circback (q\circ r)   $&$\quad$& 
(d)&$(p\circ q)\circback r   $&\ $\vdash$\ \,&$ q\circback(p\circback r)$&$\quad$&\\
(b)&$ p\circback(q\circfor r)$&\ $\vdash$\ \,&$(p\circback q)\circfor r $&$\quad$& 
(e)&$(p\circfor q)\circfor r $&\ $\vdash$\ \,&$ p\circfor (r\circ q)$&$\quad$&(II)\\
(c)&$ p\circ(q\circ r)       $&\ $\vdash$\ \,&$(p\circ q)\circ r        $&$\quad$& 
(f)&$ p\circ(q\circback r)   $&\ $\vdash$\ \,&$(q\circfor p)\circback r $&$\quad$&\\
   & &      & &           & &      & &       &\\
(a)&$ p\starfor(q\star r)    $&\ $\vdash$\ \,&$(p\starback q)\star r   $&$\quad$& 
(d)&$ p\starfor(q\star r)    $&\ $\vdash$\ \,&$(p\starfor r)\starfor q $&$\quad$&\\
(b)&$(p\star q)\star r       $&\ $\vdash$\ \,&$ p\star (q\star r)      $&$\quad$& 
(e)&$ p\starback (q\starback r)$&\ $\vdash$\ \,&$(q\star p)\starback r $&$\quad$& (III)\\
(c)&$(p\starback q)\starfor r$&\ $\vdash$\ \,&$p\starback(q\starfor r) $&$\quad$& 
(f)&$(p\starfor q)\starback r$&\ $\vdash$\ \,&$q\star (p\starback r)   $&$\quad$&\\
   & &      & &           & &      & &       &\\
(a)&$(p\circback q)\starfor r $&\ $\vdash$\ \,&$ p\circback(q\starfor r)$&$\quad$& 
(d)&$(p\circback q)\starback r$&\ $\vdash$\ \,&$ q\starback (p\circ r)  $&$\quad$&\\
(b)&$ p\circback(q\star r)    $&\ $\vdash$\ \,&$(p \circback q)\star r  $&$\quad$& 
(e)&$(p\star q)\circfor r   $&\ $\vdash$\ \,&$ p\circfor(r\starfor q)$&$\quad$&(IV)\\
(c)&$ p\circ(q\starfor r)     $&\ $\vdash$\ \,&$(p\circ q) \starfor r   $&$\quad$& 
(f)&$ p\starfor(q\starback r) $&\ $\vdash$\ \,&$(r\circfor p)\circback q$&$\quad$&\\
\end{tabular}
\end{center}

As observed in \cite[Remark 5.3]{conradie2019algorithmic},  all these axioms are analytic inductive, and hence they can all be transformed into analytic structural rules (cf.~\cite{greco2018unified}). For instance:
\[p\starback(q \,\circ\, r)\leq (p\starback q) \,\circ\, r \ \  \rightsquigarrow \ \ 
\AX$(\Delta_1 \,\hat{\starback}\, \Gamma_1) \,\hat{\circ}\, \Gamma_2 \fCenter \Delta_2$
\UI$\Delta_1 \,\hat{\starback}\, (\Gamma_1 \,\hat{\circ}\, \Gamma_2) \fCenter \Delta_2$
\DP
 \ \ \rightsquigarrow \ \ 
\AX$\Delta_1 \,\hat{\starback}\, \Gamma_1 \fCenter \Delta_2 \,\check{\circfor}\, \Gamma_2$
\UI$\Gamma_1 \,\hat{\circ}\, \Gamma_2 \fCenter \Delta_1 \,\check{\star}\, \Delta_2$
\DP
\]

By Theorem \ref{thm:semantic cut elim ax ext}, any calculus $\mathrm{D.LE}'$ obtained by adding any or more of these rules to the calculus $\mathrm{D.LE}$ for the basic $\mathcal{L}_{\mathrm{LG}}$-logic has semantic cut elimination. Moreover, in each of these rules,   the complexity of sequents does not increase from bottom to top. Hence by Proposition \ref{prop:compdecrfmp}, 
FMP holds for any $\mathrm{D.LE}'$. This captures the decidability result of \cite{moortgat2007symmetries}.

\subsection {Orthologic}\label{ssec:ortho}

The language of {\em Orthologic} (cf.\ \cite{goldblatt1974semantic}), denoted $\mathcal{L}_{\mathrm{Ortho}}$, is obtained by instantiating $\mathcal{F} = \{\sim\}$ and $\mathcal{G} = \{\neg, \abot\}$ with $n_{{\sim}} = n_{\neg} = 1, \varepsilon_{{\sim}} = \varepsilon_{\neg} = \partial$. The display calculus D.Ortho for the basic $\mathcal{L}_{\mathrm{Ortho}}$-logic contains the standard operational rules for $\sim$ and $\neg$ plus the display postulates $(\hat{\sim}, \hat{\sim}^\sharp)$ and $(\check{\neg}, \check{\neg}^\flat)$ (see section \ref{ssec:syntactic frames associated w algebras}). 

Orthologic is the axiomatic extensions of the basic $\mathcal{L}_{\mathrm{Ortho}}$-logic with the following sequents (cf.~\cite[Definition 1.1]{goldblatt1974semantic}):\footnote{Notice that the original signature of orthologic contains only one negation that is both a left and a right adjoint and, moreover, it is a self-adjoint. We have the axiom $\fneg p \dashv\vdash \gneg p$ because we decide to start with a signature $\mathcal{L}_{LE}(\mathcal{F},\mathcal{G})$ with $\mathcal{F}$ and $\mathcal{G}$ disjoint (see the definition of LE-logics and and Footnote \ref{footnote:fg}  in Section \ref{ssec:basic}).}
 $${\sim} p \dashv \vdash \neg p \qquad \abot \vdash p \qquad p \wedge \neg p \vdash \abot \qquad p \dashv\vdash \neg\neg p.$$

The axiom $\abot \vdash p$ is captured by the rule $\ABOT_W$ (see section \ref{ssec:syntactic frames associated w algebras}). The other axioms are analytic inductive, and hence, by the procedure outlined in \cite{greco2018unified}, they can be transformed into analytic structural rules:
\begin{center}
\begin{tabular}{cccc}
$\ofneg p \dashv\vdash \ogneg p$
 & $\rightsquigarrow$ & 
\AX$\Gamma \fCenter \Delta$
\UI$\OFNEG \Delta \fCenter \OGNEG \Gamma$
\DP
 & 
\AX$\OFNEG \OGNEG \Gamma \fCenter \Delta$
\UI$\Gamma \fCenter \Delta$
\DP
 \\
 & & & \\
$p \wedge \neg p \vdash \abot$
 & $\rightsquigarrow$ & 
\AX$\Gamma \fCenter \OGNEG \Gamma$
\UI$\Gamma \fCenter \ABOT$
\DP 
 & 
\AX$\Gamma_1 \fCenter \OGNEG \Gamma_2$
\doubleLine
\UI$\Gamma_2 \fCenter \OGNEG \Gamma_1$
\DP
 \\
 & & & \\
$p \dashv\vdash \neg\neg p$
 & $\rightsquigarrow$ & 
\AX$\Gamma \fCenter \OGNEG \OFNEG \Delta$
\UI$\Gamma \fCenter \Delta$
\DP
 & 
 \\
\end{tabular}
\end{center}
Let $\mathrm{D.LE}$ be the calculus for the basic $\mathcal{L}_{\mathrm{Ortho}}$-logic, and let $\mathrm{D.LE}'$ be the calculus obtained by adding the rules above to $\mathrm{D.LE}$. Theorem \ref{thm:semantic cut elim ax ext} directly applies to $\mathrm{D.LE}'$. In what follows we will show that Proposition \ref{prop:refinedfmp} can be applied to $\mathrm{D.LE}'$, by defining $\varphi'_{\mathcal{F}}$ (resp.\ $\varphi'_{\mathcal{G}}$)  as follows 
 $$\varphi'_\mathcal{F} \ceq \{(\Gamma, (\OFNEG\OGNEG)^{n} \Gamma), ((\OFNEG\OGNEG)^{m} \Gamma, \Gamma): n, m\in\mathbb{N} \mbox{ and } \Gamma \in \mathsf{Str}_\mathcal{F} \},$$ 
 $$\varphi'_\mathcal{G} \ceq \{(\Delta, (\OGNEG\OFNEG)^{n} \Delta), ((\OGNEG\OFNEG)^{m} \Delta, \Delta): n, m\in\mathbb{N} \mbox{ and } \Delta \in \mathsf{Str}_\mathcal{G} \}$$
Clearly, $\varphi_\mathcal{F}$ and $\varphi_\mathcal{G}$ are congruences. The applicability of Proposition \ref{prop:refinedfmp} is an immediate consequence of the following.
\begin{lem}
	\begin{enumerate}
		\item $\varphi'_\mathcal{F}\subseteq\varphi_\mathcal{F}$.
		\item For every sequent $\Gamma \Rightarrow \Delta$ the set $\{[\Gamma']_{\varphi'_{\mathcal{F}}}\mid \Gamma'\in (\Gamma \Rightarrow \Delta)^{\leftarrow}_{\mathcal{F}}\}$ is finite.
	\end{enumerate}
\end{lem}
\begin{proof}
	1. By $m$ consecutive applications of the rule 
	\begin{center}
	\AX$\OFNEG\OGNEG \Gamma \fCenter \Delta$
	\UI$\Gamma \fCenter \Delta$
	\DP	
	\end{center}
we obtain the derivability of the following rule
 \begin{center}
	\AX${(\OFNEG\OGNEG)^{m}\,} \Gamma \fCenter \Delta$
	\UI$\Gamma \fCenter \Delta$
	\DP	
\end{center}
Likewise, by $n$ consecutive applications of the following sequence of rules 
\begin{center}
\AX$\Gamma \fCenter \Delta$
\UI$\OFNEG \Delta \fCenter \OGNEG \Gamma$
\UI$\OFNEG\OGNEG \Gamma \fCenter \Delta$
\DP
\end{center}
we obtain the derivability of the following rule
\begin{center}
\AX$\Gamma \fCenter \Delta$
\UI${(\OFNEG\OGNEG)^{n}\,} \Gamma \fCenter \Delta$
\DP	
\end{center}
2. Fix a sequent $\Gamma \Rightarrow \Delta$. It is enough to show that if $\Pi\in(\Gamma\Rightarrow \Delta)^{\leftarrow}_{\mathcal{F}}\cup(\Gamma \Rightarrow \Delta)^{\leftarrow}_{\mathcal{G}}$ then $(\Pi,\Pi')\in\varphi'_{\mathcal{F}}\cup\varphi'_{\mathcal{G}}$ for some structure $\Pi'$ belonging to the following finite set: 
\begin{align*}\Sigma \ceq & \ \mathrm{SUB}(\Gamma, \Delta)\,\cup\,\mathrm{sub}(\Gamma, \Delta)\,\cup\,
\\&\OFNEG(\mathrm{SUB}(\Gamma, \Delta)\,\cup\,\mathrm{sub}(\Gamma, \Delta))\,\cup\,\OGNEG(\mathrm{SUB}(\Gamma, \Delta)\,\cup\,\mathrm{sub}(\Gamma, \Delta))\,\cup\,\{\ABOT,\OFNEG \ABOT\},
\end{align*}
 where $\mathrm{SUB}(\Psi)$ is the set of substructures of $\Psi$, $\mathrm{sub}(\Psi)$ is the set of subformulas of formulas in $\mathrm{SUB}(\Psi)$, $\OGNEG X=\{\OGNEG \Psi \mid \Psi \in X\}$ and $\OFNEG X=\{\OFNEG \Psi \mid \Psi \in X\}$ for any set of structures $X$. We proceed by induction on the inverse proof-trees. The base case, i.e.~$\Pi\in\{\Gamma,\Delta\}$, is clear. As to the inductive step, the proof proceeds by inspection on the rules. The cases regarding applications of introduction rules or structural rules of $\mathrm{D.LE}$ which reduce the complexity of sequents when applied bottom-up are straightforward and omitted. Let $\Lambda \Rightarrow \Theta \in (\Gamma \Rightarrow \Delta)^{\leftarrow}$ and assume that $(\Lambda,\Lambda'),(\Theta,\Theta')\in\varphi'_{\mathcal{F}}\cup\varphi'_{\mathcal{G}}$ for some $\Lambda',\Theta'\in\Sigma$. Then, the bottom-up application of one of the following rules 
\[\AX$\OFNEG\OGNEG \Gamma \fCenter \Delta$
\UI$\Gamma \fCenter \Delta$
\DP
\quad
\AX$\Gamma \fCenter \OGNEG \OFNEG \Delta$
\UI$\Gamma \fCenter \Delta$
\DP
 \quad 
\AX$\Gamma \fCenter \ONEG \Gamma$
\UI$\Gamma \fCenter \ABOT$
\DP
 \quad 
\AX$\Gamma \fCenter \ABOT$
\UI$\Gamma \fCenter \Delta$
\DP\]
to $\Lambda \Rightarrow \Theta$ yields $\OFNEG\OGNEG \Lambda \Rightarrow \Theta$, $\Lambda \Rightarrow \OGNEG\OFNEG \Theta$, $\Lambda \Rightarrow \ONEG \Lambda$ and $\Lambda \Rightarrow \ABOT$ respectively. Hence, $(\OFNEG\OGNEG \Lambda,\Lambda),(\Lambda,\Lambda')\in\varphi'_{\mathcal{F}}$, $(\OGNEG\OFNEG \Theta,\Theta),(\Theta,\Theta')\in\varphi'_{\mathcal{G}}$ and therefore $(\OFNEG\OGNEG \Lambda,\Lambda')\in\varphi'_{\mathcal{F}}$, $(\OGNEG\OFNEG \Theta,\Theta')\in\varphi'_{\mathcal{G}}$, $(\OGNEG \Lambda, \OGNEG \Lambda')\in\varphi'_{\mathcal{G}}$, $\ABOT\in\Sigma$. Finally, if $\Lambda'\in \mathrm{SUB}(\Gamma, \Delta) \cup \mathrm{sub}(\Gamma, \Delta)$ then $\OGNEG \Lambda \in \Sigma$, and if $\Lambda'\in \OFNEG(\mathrm{SUB}(\Gamma, \Delta) \cup \mathrm{sub}(\Gamma, \Delta))$ then $\Lambda'$ is $\OFNEG \Lambda''$ where $\Lambda'' \in \Sigma$. Therefore $(\OGNEG \Lambda, \OGNEG \OFNEG \Lambda'') \in \varphi'_{\mathcal{G}}$ and $(\OGNEG \OFNEG \Lambda, \Lambda'') \in \varphi'_{\mathcal{G}}$, so $(\OGNEG \Lambda, \Lambda'') \in \varphi'_{\mathcal{G}}$.
\end{proof}

\section{Conclusions and further directions}\label{sec:concl}
\paragraph{Contributions of the present paper.} This paper extends the research programme in algebraic proof theory from substructural logics to arbitrary normal LE-logics. Our original contributions concern, on the proof-theoretic side, the use of display calculi in the context of algebraic proof theory; on the algebraic side, the use of the canonical extension construction and the constructive canonicity of analytic inductive inequalities \cite{conradie2016constructive}. 

\paragraph{LE-logics as generalized modal logics.} Our use of canonical extensions (see paragraph above) reflects the fact that the results of the present paper are grounded on insights which derive from duality theory in modal logic. To emphasize this aspect, we use terminology which is closer to the literature in modal logic than to the literature in universal algebra. The results of the present paper pertain to a research strand which stems from the generalized Sahlqvist theory for LE-logics \cite{conradie2019algorithmic} and has given rise not
only to the canonicity results which are at the heart of the semantic cut
elimination of the present paper, but also to a systematic way of
defining various types of semantics for LE-logics \cite{conradie2016categories,CFPPTW17,ModellingInformationalEntropy,VectorSpacesAsKripkeFrames} and their {\em many-valued} versions \cite{GJMPT19,RoughConcepts,ModellingCompetingTheories,TheLogicOfVagueCategories,ModellingSocioPoliticalCompetition}
which are also connected to (generalized) probabilistic logics \cite{TowardsaDempsterShaferTheoryOfConcepts}. Not
only do these lines of research have a common root with the ones of the present paper, but point at the possibility to create an overarching
mathematical environment in which results such as canonicity, semantic
cut elimination and the Goldblatt-Thomason theorem \cite{CPT18} can be formulated
and proved in a uniform and parametric way for two-valued and
many-valued logics, and 
 in which duality-theoretic results, algebraic results and proof-theoretic results are used in synergy rather than in isolation. 

\paragraph{MacNeille completions and complex algebras.} Key to our results is the notion of functional D-frame, obtained as the direct generalization of residuated frames in \cite{galatos2013residuated}. The name emphasizes that the relation associated with the additional connectives $f\in\mathcal{F}$ and $g\in\mathcal{G}$ are functional (cf.\ Section \ref{sec: functional L-frames}). Then the construction corresponding to taking the MacNeille completion of a  residuated frame is the complex algebra construction based on the concept lattice associated with any polarity. Further directions involve making use of this insight towards the generalization of the characterization results in \cite{ciabattoni2012algebraic} to the general setting of normal LE-logics.

\paragraph{More metalogical properties via algebraic proof theory.} A natural prosecution of this research programme concerns uniforms proofs of metatheoretic properties of (classes of) LE-logics, such as finite embeddability property, disjunction property, Craig interpolation. These results typically lie at the interface between algebra and proof theory. On the proof-theoretic side, the present setting provides a platform for testing the potential of display calculi in obtaining results which are typically obtained via Gentzen calculi  (cf.\ \cite{ono1998proof,galatos2010cut,galatos2013residuated}). On the algebraic side, the present setting allows to extract the essentials of very elegant and meaningful proofs obtained in the literature (cf.\ \cite{blok2002finite,blok2005finite,galatos2006algebraization,souma2007algebraic,kihara2008algebraic,kihara2009interpolation,horvcik2011disjunction,maksimova2012amalgamation}) for specific signatures and make them independent of a specific language. 

\paragraph{Projection onto distributive LE-logics.} The present setting smoothly addresses the semantic cut-elimination for distributive LE-logics, i.e.~logics algebraically captured by varieties of normal lattice expansions the lattice reducts of which are distributive. Key to this is the observation that any binary fusion-type (resp.\ fission-type) connective for which the analytic structural rules weakening, exchange and contraction hold must coincide with conjunction (resp.\ disjunction). However the general FMP result does not directly apply because contraction is a prime example of a rule violating the assumption of Proposition \ref{prop:compdecrfmp}. We are currently investigating whether a more sophisticated route allows us to encompass FMP for classes of distributive LE-logics. 

\bibliography{bib}
\bibliographystyle{plain}

\appendix

\section{Proper display calculi  and analytic structural rules}	\label{ap:display}
In this section, we recall the definition of analytic structural rules which is introduced in \cite{greco2018unified}. This definition is tightly connected with the notion of  {\em proper display calculus} (cf.~\cite{Wa98}), since it is aimed at guaranteeing that adding an analytic structural rule to a proper display calculus preserves cut elimination and subformula property. 


First of all let us recall some terminology (see e.g.~\cite[Section 4.1]{Wa98}) and notational conventions. A {\em sequent}, also called a consecution in the display calculi literature, is a pair $\Gamma \fCenter \Delta$ where $\Gamma$ and $\Delta$ are structures, $\Gamma$ is called the {\em antecedent} and $\Delta$ is called the {\em consequent} of the sequent. An {\em inference} $r$, also called an instance of a rule, is a pair $(S, s)$ of a (possibly empty) set of sequents $S$ (the premises) and a sequent $s$ (the conclusion). A {\em rule} (of inference) $R$, also called a scheme, is a schematic inference using meta-variables for structures (in this paper we use capital-case Greek letters) or meta-variables for formulas (in this paper we use lower-case Greek letters). We identify a rule $R$ with the set of all inferences that are instantiations of $R$. A rule $R$ with no premises, i.e.~$S = \emptyset$, is called an {\em axiom scheme}, and an instantiation of such $R$ is called an {\em axiom}. The immediate subformulas of a principal formula (see Definition \ref{def:ParametersCongruenceHistoryTree}) in the premise(s) of an operational inference are called {\em auxiliary formulas}. The formulas that are not preserved in an inference instantiating the cut rule are called {\em cut formulas}. If the cut formulas are principal in an inference instantiating the cut rule, then the inference is called {\em principal cut}. A {\em proof} of (the instantiation of) a sequent $\Gamma \fCenter \Delta$ is a tree where (the instantiation of) $\Gamma \fCenter \Delta$ appears as the end-sequent, all the leaves are axioms, and each node is introduced via an inference. 
\begin{definition}
A proof system enjoys the {\em display property} if for every sequent $\Gamma \fCenter \Delta$ and every substructure $\Sigma$ of either $\Gamma$ or $\Delta$, the sequent $\Gamma \fCenter \Delta$ can be equivalently transformed, using the rules of the system, into a sequent which is either of the form $\Sigma \fCenter \Pi$ or of the form $\Pi \fCenter \Sigma$, for some structure $\Pi$. In the first case, $\Sigma$ is displayed in {\em precedent position}, and in the second case, $\Sigma$ is displayed in {\em succedent position}. The rules enabling this equivalent rewriting are the {\em display postulates}. 
\end{definition}
\begin{remark}
In other words, in a calculus enjoying the display property, any substructure of any sequent $\Gamma \fCenter \Delta$ is always displayed either only in precedent position or only in succedent position. This is why we can talk about occurrences of substructures in precedent or in succedent position, even if they are nested deep within a given sequent. If a structure $\Sigma$ occurs in the sequent $\Gamma \fCenter \Delta$ in precedent (resp.~succedent) position we write $(\Gamma \fCenter \Delta)[\Sigma]^{pre}$ (resp.~$(\Gamma \fCenter \Delta)[\Sigma]^{suc}$).
\end{remark}
	
An {\em analysis} of the rules provides the notions of `parameters' and `congruence' and makes it possible to decide whether a structure (resp.~formula) is preserved or introduced in an inference in a principled way. `Intuitively, parameters are structure occurrences which are either held constant from premises to conclusion or are introduced with no regard to their shape (e.g.~structures introduced by weakening are considered parameters)', and the congruence relation is meant to identify the different occurrences of the same substructure along the branches of a proof \cite[Section 4]{Belnap}, \cite[Definition 6.5]{Restall}. In what follows we provide a formal definition of these intuitive notions.\footnote{To the best of our knowledge, a completely rigorous definition of these intuitive notions was never provided in the display calculi literature. \cite[Definition 6.5]{Restall} and the terminology recalled in \cite[Section 4.1]{Wa98} rephrase the definition of congruent parameters proposed by Belnap in his 1987 seminal paper \cite[Definition 4.1]{Belnap}: ``Constituents occupying similar positions in occurrences of structures assigned to the same structure-variable are defined {\em congruent} in {\em Inf}''. We notice that the expression `similar positions' is not formally defined and we observe that a formal definition presupposes condition C$_2$.} 
\begin{definition}
\label{def:ParametersCongruenceHistoryTree}
{\em Specifications} are instantiations of structure meta-variables in the statement of a rule $R$. The {\em parameters} of $r \in R$ are {\em substructures} of instantiations of structure meta-variables in the statement of $R$. A formula instance is {\em principal} in an inference $r \in R$ if it is not a parameter in the conclusion of $r$. Structure occurrences in an inference $r \in R$ are in the (symmetric) relation of {\em local congruence} if they instantiate the same meta-variable in the statement of $R$. Therefore, the local congruence is a relation between specifications.
\end{definition}

We are ready now to recall the conditions C$_1$-C$_8$ defining a {\em proper display calculus}. Our presentation closely follows \cite[Section 2.2]{greco2018unified}.

	\paragraph*{ C$_1$: Preservation of formulas.} This condition requires each formula occurring in a premise of a given inference to be a subformula of some formula in the conclusion of that inference. 
This condition is not included in the list of sufficient conditions of the cut elimination metatheorem (see Theorem \ref{thm:meta}), but, in the presence of cut elimination, it guarantees the \emph{subformula property} of a proof system.
	Considering the rules of $\mathrm{D.LE}$ given above, condition C$_1$ can be verified by inspection, noting that only the Cut rule fails C$_1$. In practice, condition C$_1$ also prohibits rules in which structure variables occurring in some premise do not occur as well in the conclusion, since in concrete derivations these are typically instantiated with (structures containing) formulas which would then disappear in the application of the rule.
	\paragraph*{ C$_2$: Shape-alikeness.}
	This condition is based on the relation of local congruence between specifications in inferences. 
	Condition C$_2$ requires that locally congruent specifications are occurrences of the same structure. This can be understood as a condition on the {\em design} of the rules of the system if the local congruence relation is understood as part of the definition of each given rule; that is, each schematic rule of the system comes with an explicit description of which elements are locally congruent to which via a notational convention. 
	In this respect, C$_2$ is nothing but a sanity check, requiring that the local congruence is defined in such a way that it indeed identifies the occurrences which are intuitively ``the same''.\footnote{Our convention throughout the paper is that congruent parameters are denoted by the same letter. For instance, in the rule $$\frac{\Gamma \circ \Gamma' \vdash \Delta}{\Gamma' \circ \Gamma \vdash \Delta}$$ the structures $\Gamma \circ \Gamma'$ and $\Delta$ are parametric and the occurrences of $\Gamma$ (resp.~$\Gamma'$, $\Delta$) in the premise and the conclusion are congruent.} 

Condition C$_2$ guarantees that locally congruent specifications have the same generation tree, therefore we can give the following:

\begin{definition}
Two parameters are in the {\em inference congruence relation} if they correspond to the same subtree of the generation tree of two locally congruent specifications. The {\em proof congruent relation} is the transitive closure of the {\em inference congruence relation}. 
\end{definition}
	
	\paragraph*{C$_3$: Non-proliferation.} Like the previous one, also this condition is actually about the definition of the local congruence relation on specifications. Condition C$_3$ requires that, for every inference, each of its specifications is locally congruent to at most one specification in the conclusion of that inference. Hence, the condition stipulates that for a rule such as the following is de facto banned
	\begin{center}
		\AX$\Gamma \fCenter \Delta$
		\UI$\Gamma \circ \Gamma \fCenter \Delta$
		\DisplayProof
	\end{center}
\noindent because, given the notational conventions on the design of rules, the structure $\Gamma$ from the premise is declared locally congruent to two occurrences of $\Gamma$ in the conclusion sequent. Vice versa, the following rule is acceptable 
\begin{center}
		\AX$\Gamma \fCenter \Delta$
		\UI$\Gamma' \circ \Gamma \fCenter \Delta$
		\DisplayProof
	\end{center}
\noindent because the specification of $\Gamma$ (resp.~$\Delta$) from the premise is declared locally congruent to the specification of $\Gamma$ (resp.~$\Delta$) in the conclusion, and the specification of $\Gamma'$ in the conclusion is declared locally congruent to itself. 
In practice, in the general schematic formulation of rules, we will use the same structure meta-variable for two different occurrences of specifications if and only if they are locally congruent. 
		
	\paragraph*{C$_4$: Position-alikeness.} This condition bans any rule in which a (sub)structure in precedent (resp.~succedent) position in a premise is locally congruent to a (sub)structure in succedent (resp.\ precedent) position in the conclusion.
	
	\bigskip
	
	\paragraph*{C$_5$: Display of principal constituents.} This condition requires that any principal occurrence (that is, a non-parametric formula occurring in the conclusion of a rule application, cf.\ \cite[Condition C5]{Belnap}) be always either the entire antecedent or the entire consequent part of the sequent in which it occurs.
	
	\bigskip
	The following conditions C$_6$ and C$_7$ are not reported below as they are stated in the original paper \cite{Belnap}, but as they appear in \cite[Section 4.1]{Wa98}. 
	%
	
	\paragraph*{C$_6$: Closure under substitution for succedent parameters.} This condition requires each rule to be closed under simultaneous substitution of arbitrary structures for formulas which are congruent parameters occurring in succedent position.
	Condition C$_6$ ensures, for instance, that if the following inference is an application of the rule $R$:
	
	\begin{center}
		\AX$(\Gamma \fCenter \Delta) \big([\varphi]^{suc}_{i} \,|\, i \in I\big)$
		\RightLabel{$R$}
		\UI$(\Gamma' \fCenter \Delta') [\varphi]^{suc}$
		\DisplayProof
	\end{center}
	
	\noindent and $\big([\varphi]^{suc}_{i} \,|\, i \in I\big)$ represents all and only the occurrences of $\varphi$ in the premiss which are congruent to the occurrence of $\varphi$ in the conclusion\footnote{Clearly, if $I = \varnothing$, then the occurrence of $\varphi$ in the conclusion is congruent to itself.}, 
	then also the following inference is an application of the same rule $R$:
	
	\begin{center}
		\AX$(\Gamma \fCenter \Delta) \big([\Sigma/\varphi]^{suc}_{i} \,|\, i \in I\big)$
		\RightLabel{$R$}
		\UI$(\Gamma' \fCenter \Delta') [\Sigma/\varphi]^{suc}$
		\DisplayProof
	\end{center}
	
	\noindent where the structure $\Sigma$ is substituted for $\varphi$.
	
	\noindent This condition caters for the step in the cut elimination procedure in which the cut needs to be ``pushed up'' over rules in which the cut-formula in succedent position is parametric. Indeed, condition C$_6$ guarantees that, in the picture below, a well-formed subtree $\pi_1[\Delta/\varphi]$ can be obtained from $\pi_1$ by replacing any occurrence of $\varphi$ corresponding to a node in the history tree of the cut-formula $\varphi$ by $\Delta$, and hence the following transformation step is guaranteed go through uniformly and ``canonically'':
	
	\begin{center}
		\bottomAlignProof
		\begin{tabular}{lcr}
			\AXC{\ \ \ $\vdots$ \raisebox{1mm}{$\pi'_1$}}
			\noLine
			\UI$\Gamma' \fCenter \varphi$
			\noLine
			\UIC{\ \ \ $\vdots$ \raisebox{1mm}{$\pi_1$}}
			\noLine
			\UI$\Gamma \fCenter \varphi$
			\AXC{\ \ \ $\vdots$ \raisebox{1mm}{$\pi_2$}}
			\noLine
			\UI$\varphi \fCenter \Delta$
			\BI$\Gamma \fCenter \Delta$
			\DisplayProof
			& $\rightsquigarrow$ &
			\bottomAlignProof
			\AXC{\ \ \ $\vdots$ \raisebox{1mm}{$\pi'_1$}}
			\noLine
			\UI$\Gamma' \fCenter \varphi$
			\AXC{\ \ \ $\vdots$ \raisebox{1mm}{$\pi_2$}}
			\noLine
			\UI$\varphi \fCenter \Delta$
			\BI$\Gamma' \fCenter \Delta$
			\noLine
			\UIC{\ \ \ \ \ \ \ \ \ $\vdots$ \raisebox{1mm}{$\pi_1[\Delta/\varphi]$}}
			\noLine
			\UI$\Gamma \fCenter \Delta$
			\DisplayProof
			\\
		\end{tabular}
	\end{center}
	\noindent if each rule in $\pi_1$ verifies condition C$_6$.
	
	\paragraph*{C$_7$: Closure under substitution for precedent parameters.} This condition requires each rule to be closed under simultaneous substitution of arbitrary structures for formulas which are congruent parameter occurring in precedent position.
	Condition C$_7$ can be understood analogously to C$_6$, relative to formulas in precedent position. Therefore, for instance, if the following inference is an application of the rule $R$:
	
	\begin{center}
		\AxiomC{$(\Gamma \fCenter \Delta) \big([\varphi]^{pre}_{i} \,|\, i \in I\big)$}
		\RightLabel{$R$}
		\UnaryInfC{$ (\Gamma' \fCenter \Delta') [\varphi]^{pre}$}
		\DisplayProof
	\end{center}
	
	\noindent then also the following inference is an instance of $R$:
	
	\begin{center}
		\AxiomC{$(\Gamma \fCenter \Delta) \big([\Sigma/\varphi]^{pre}_{i} \,|\, i \in I\big)$}
		\RightLabel{$R$}
		\UnaryInfC{$ (\Gamma' \fCenter \Delta') [\Sigma/\varphi]^{pre}$}
		\DisplayProof
	\end{center}
	
	\noindent Similarly to what has been discussed for condition C$_6$, condition C$_7$ caters for the step in the cut elimination procedure in which the cut needs to be ``pushed up'' over rules in which the cut-formula in precedent position is parametric.
	
	\paragraph*{C$_8$: Eliminability of matching principal constituents.}
	
	This condition  requests a standard Gentzen-style checking, which is now limited to the case in which both cut formulas are {\em principal}, i.e.\ each of them has been introduced with the last rule application of each corresponding subdeduction. In this case, analogously to the proof Gentzen-style, condition C$_8$ requires being able to transform the given deduction into a deduction with the same conclusion in which either the cut is eliminated altogether, or is transformed into one or more applications of cut involving proper subformulas of the original cut-formulas.
	
\begin{definition}
The {\em history of a principal formula} is a single node labelled by the formula. 
The {\em history of a structure in a proof} is the intersection of the proof congruence relation with the proof tree relation, and each node is labelled by the structure that generates the proof congruence relation.\footnote{Notice that this can be made formally precise by observing that the instances of substructures in each sequent inherits the strict order of the proof tree.}
\end{definition}

	\begin{remark}
		\label{PS:rem: history tree}
Notice that C$_3$ implies that the history of any structure within a given proof has the shape of a tree, therefore we use `history' and `history-tree' interchangeably. Notice, however, that the history-tree of a structure might have a different shape than the portion of the underlying proof corresponding to it; for instance, the following application of the Contraction rule gives rise to a bifurcation of the history-tree of $\varphi$ which is absent in the underlying branch of the proof tree, given that Contraction is a unary rule.

\ \ \ \ \ \ \ \ \ \ \ \ \ \ \ \ \ \ \ \ \ \ \ \ 
		\begin{tabular}{@{}lcr}
			\bottomAlignProof
			\AXC{$\vdots$}
			\noLine
			\UI$\varphi \circ \varphi \fCenter \Delta$
			\UI$\varphi \fCenter \Delta$
			\DisplayProof
			
			& \ \ \ \ \ \ \ \ &
			
			$\unitlength=0.80mm
			\put(0.00,10.00){\circle*{2}}
			\put(10.00,0.00){\circle*{2}}
			\put(20.00, 10.00){\circle*{2}}
			\put(40.00, 10.00){\circle*{2}}
			\put(0.00,10.00){\line(1,-1){10}}
			\put(10.00,0.00){\line(1,1){10}}
			\put(40.00, 0.00){\circle*{2}}
			\put(40.00,0.00){\line(0,1){10}}
			$
			\\
		\end{tabular}

	\end{remark}

In any display calculus, principal formulas are introduced either by axioms or by operational inferences. The definition of congruent relation and conditions C$_2$, C$_3$ and C$_4$ guarantee that: (i) the history tree of a structure in a proof exists and it is unique, (ii) it is a tree, (iii) the label of each node exists and it is unique. 

To exemplify the notions introduced so far, we consider a concrete LE-logic. The language of the basic {\em tense logic}, denoted $\mathcal{L}_{\mathrm{T}}$, is obtained by instantiating $\mathcal{F}: = \{\wdia\}$ and $\mathcal{G} = \{\bbox\}$ with $n_\wdia = n_\bbox = 1$ and $\varepsilon_\wdia =\varepsilon_\bbox = 1$. The following is a proof of $\wdia (q \aor r) \fCenter \wdia q \aor \wdia r$ in $\mathrm{D.L_T}$.

\begin{center}
\AXC{\ }
\LL{\fns Id}
\UI$q \fCenter q$
\RL{\fns $\wdia_R$}
\UI$\WDIA q \fCenter \wdia q$
\RL{\fns $\aor_{R1}$}
\UI$\WDIA q \fCenter \wdia q \aor \wdia r$
\LL{\fns $\WDIA \dashv \BBOX$}
\UI$q \fCenter \BBOX (\wdia q \aor \wdia r)$
\AXC{\ }
\LL{\fns Id}
\UI$r \fCenter r$
\RL{\fns $\wdia_R$}
\UI$\WDIA r \fCenter \wdia r$
\RL{\fns $\aor_{R2}$}
\UI$\WDIA r \fCenter \wdia q \aor \wdia r$
\LL{\fns $\WDIA \dashv \BBOX$}
\UI$r \fCenter \BBOX (\wdia q \aor \wdia r)$
\LL{\fns $\aor_L$}
\BI$q \aor r \fCenter \BBOX (\wdia q \aor \wdia r)$
\RL{\fns $\WDIA \dashv \BBOX$}
\UI$\WDIA (q \aor r) \fCenter \wdia q \aor \wdia r$
\LL{\fns $\wdia_L$}
\UI$\wdia (q \aor r) \fCenter \wdia q \aor \wdia r$
\DP
\end{center}


In order to see the notions introduced so far at work, we build step by step the history tree of the formula $\wdia q \aor \wdia r$ occurring in the end-sequent of the proof above.

Consider the subtree on the left obtained via the inference $r$ falling under the rule $R = \wdia_L$ that we recall on the right, 
\begin{center}
\begin{tabular}{cc}
\AX$\WDIA (q \aor r) \fCenter \wdia q \aor \wdia r$
\LL{\fns $r \in R$}
\UI$\wdia (q \aor r) \fCenter \wdia q \aor \wdia r$
\DP
 & 
\AX$\WDIA \varphi \fCenter \Delta$
\LL{\fns $R = \wdia_L$}
\UI$\wdia \varphi \fCenter \Delta$
\DP
 \\
\end{tabular}
\end{center}
\noindent The formula occurrence $\wdia q \aor \wdia r$ in the conclusion (resp.~premise) of $r \in R$ instantiates the structure meta-variable $\Delta$ in the conclusion (resp.~premise) of $R$, therefore it is a parameter. Moreover, all the occurrences of the formula $\wdia q \aor \wdia r$ in $r$ are instances of the same structure meta-variable $\Delta$ in $R$, therefore they are locally congruent. The formula occurrence $\wdia (q \aor r)$ in the conclusion of $r$ does not instantiate a structure meta-variable in the conclusion of $R$, therefore it is non-parametric and it counts as a principal formula. The structure occurrence $\WDIA (q \aor r)$ in the premise of $r$ is an instantiation of the structure $\WDIA \varphi$ in the premise of $R$, but it is not the instantiation of a structure meta-variable, therefore, at this stage, we cannot declare $\WDIA (q \aor r)$ a parameter (it is not principal because it is a structure and it occurs in a premise). Analogously, the formula occurrence $q \aor r$ in the premise of $r$ does not instantiate a structure meta-variable, therefore, at this stage, we cannot declare $q \aor r$ a parameter (it is not principal because it occurs in a premise). 

Consider the subtree on the left obtained via the inference $r$ falling under the rule $R = \WDIA \dashv \BBOX$ that we recall on the right, 
\begin{center}
\begin{tabular}{cc}
\AX$q \aor r \fCenter \BBOX (\wdia q \aor \wdia r)$
\RL{\fns $r \in R$}
\UI$\WDIA (q \aor r) \fCenter \wdia q \aor \wdia r$
\DP
 & 
\AX$\Gamma \fCenter \BBOX \Delta$
\RL{\fns $R = \WDIA \dashv \BBOX$}
\UI$\WDIA \Gamma \fCenter \Delta$
\DP
 \\
\end{tabular}
\end{center}
\noindent The formula occurrences $q \aor r$ in the premise and, respectively, in the conclusion of $r \in R$ are locally congruent because they are instances of the same structure meta-variable $\Gamma$ occurring in the premise and, respectively, in the conclusion of $R$. Likewise the formula occurrences $\wdia q \aor \wdia r$ in $r \in R$ are locally congruent because they instantiate the same meta-variable $\Delta$ in $R$. 

Consider the subtree on the left obtained via the inference $r$ falling under the rule $R = \WDIA \dashv \BBOX$ that we recall on the right, 
\begin{center}
\begin{tabular}{cc}
\AX$q \fCenter \BBOX (\wdia q \aor \wdia r)$
\AX$r \fCenter \BBOX (\wdia q \aor \wdia r)$
\LL{\fns $r \in R$}
\BI$q \aor r \fCenter \BBOX (\wdia q \aor \wdia r)$
\DP
 & 
\AX$p \fCenter \Delta$
\AX$q \fCenter \Delta$
\LL{\fns $R = \aor_L$}
\BI$p \aor q \fCenter \Delta$
\DP
 \\
\end{tabular}
\end{center}
\noindent The formula occurrence $q \aor r$ in the conclusion of $r$ does not instantiate a structure meta-variable in the conclusion of $R$, therefore it is a principal formula. The structure occurrences $\BBOX (\wdia q \aor \wdia r)$ in $r \in R$ are locally congruent because they instantiate the same meta-variable $\Delta$ in $R$. The formula occurrences $\wdia q \aor \wdia r$ in $r \in R$ are congruent because they instantiate the same substructure of the $\BBOX (\wdia q \aor \wdia r)$ in $R$.

Considering the backward application of the rules $R = \WDIA \dashv \BBOX$ and $R = \aor_{R}$ in each branch, we will reach the sequent where $\wdia q \aor \wdia r$ is principal and it is easy to check that the only subtree of the proof containing the congruence class of $\wdia q \aor \wdia r$ is the following

\begin{center}
\def\fCenter{{\mbox{$\ \phantom{\Rightarrow}\ $}}}

\AXC{\ }
\RL{\fns $\aor_{R1}$}
\UI$\phantom{\WDIA q} \fCenter \wdia q \aor \wdia r$
\LL{\fns $\WDIA \dashv \BBOX$}
\UI$\phantom{q} \fCenter \phantom{\BBOX} \wdia q \aor \wdia r$

\AXC{\ }
\RL{\fns $\aor_{R2}$}
\UI$\phantom{\WDIA r} \fCenter \wdia q \aor \wdia r$
\LL{\fns $\WDIA \dashv \BBOX$}
\UI$\phantom{r} \fCenter \phantom{\BBOX} \wdia q \aor \wdia r$
\LL{\fns $\aor_{L}$}
\BI$\phantom{q \aor r} \fCenter \phantom{\BBOX} \wdia q \aor \wdia r$
\RL{\fns $\WDIA \dashv \BBOX$}
\UI$\phantom{\WDIA (q \aor r)} \fCenter \wdia q \aor \wdia r$
\LL{\fns $\wdia_{L}$}
\UI$\phantom{\wdia (q \aor r)} \fCenter \wdia q \aor \wdia r$
\DP
\def\fCenter{{\mbox{$\ \Rightarrow\ $}}}
\end{center}

	\begin{thm} (cf.\ \cite[Section 3.3, Appendix A]{Wan02})
		\label{thm:meta}
		Any calculus satisfying conditions C$_2$, C$_3$, C$_4$, C$_5$, C$_6$, C$_7$, C$_8$ enjoys cut elimination. If C$_1$ is also satisfied, then the calculus enjoys the subformula property.
	\end{thm}

\begin{definition}[Analytic structural rules]\label{def:analytic}
		(cf.\ \cite[Definition 3.13]{CiRa14}) A structural rule which satisfies conditions C$_1$-C$_7$ is an \emph{analytic structural rule}.
	\end{definition}
	
\begin{prop}(cf.~\cite{greco2018unified})
	\label{prop:type5} Every analytic $(\Omega, \epsilon)$-inductive LE-inequality can be equivalently transformed, via an ALBA-reduction, into a set of analytic structural rules.
\end{prop}

Below we briefly check that $\mathrm{D.LE^*}$ is a proper display calculus.
\begin{itemize}
\item In the operational rules of $\mathrm{D.LE^*}$ a formula meta-variable is not preserved from premise(s) to conclusion, but these formula meta-variables occur as subformulas of the principal formula. Of course, also Cut does not preserve formula meta-variables from premises to conclusion, but it is eliminable. Moreover, in all the rules of $\mathrm{D.LE^*}$ the structure meta-variable are preserved from premise(s) to conclusion.\footnote{The symbol $\AATOP$ in the rule $\AATOP_W$ (resp.~$\ABOT$ in the rule $\ABOT_W$) is not preserved from premise to conclusion, but $\AATOP$ (resp.~$\ABOT$) is a zeroary connective, it is not a structure meta-variable. The only way to introduce $\AATOP$ (resp.~$\ABOT$) is via the rule $\top_R$ (resp.~$\bot_L$), so $\top$  (resp.~$\bot$) occurs as a subformula in the conclusion of $\AATOP_W$ (resp.~$\ABOT_W$).} Therefore, condition C$_1$ is satisfied. 

\item All the rules of $\mathrm{D.LE^*}$ implement the notational convention that structures are considered congruent in an inference only if they are instances of the same structure meta-variable. Therefore, condition C$_2$ is satisfied. 

\item In each rule of $\mathrm{D.LE^*}$ a structure meta-variable occurs at most once in the premise and at most once in the conclusion. Therefore, condition C$_3$ is trivially satisfied. 

\item The language of $\mathrm{D.LE^*}$ makes use of two disjoint sets of structural meta-variables: $\Gamma \in \mathsf{Str}_\mathcal{F}$ and $\Delta \in \mathsf{Str}_\mathcal{G}$. In each rule of $\mathrm{D.LE^*}$ structures meta-variables $\Gamma$ occur only in precedent position and structures meta-variables $\Delta$ occur only in succedent position. Therefore, condition C$_4$ is immediately satisfied. 

\item In $\mathrm{D.LE^*}$, the identity axiom schema introduce a single principal (atomic) formula in precedent position and a single principal (atomic) formula in succedent position, and operational rules introduce a single principal (complex) formula either in precedent or in succedent position; moreover, principal formulas are introduced {\em in display}, e.g.~they occur as the entire structure either in precedent or in succedent position. Therefore, condition C$_5$ is satisfied. 

\item The only structural rules of $\mathrm{D.LE^*}$ are Cut, the display postulates, and the rules $\AATOP_W$ and $\ABOT_W$. All the structure meta-variables occurring in these rules are arbitrary and no side conditions is used. Therefore, condition C$_6$ and C$_7$ are satisfied. 

\item The eliminability of principal cuts whenever the cut formulas are atoms or constants is immediate given that one of the premises and the conclusion of the cut are the same sequent. If the cut formulas are complex, a standard proof transformation reducing the complexity of the cut is always available. Below we exemplify the proof transformation in the case of a connective $g$ with $n_g = 3$, $\varepsilon_{g, 1} = \partial$, $\varepsilon_{g, 2} = \varepsilon_{g, 3} = 1$. 
\end{itemize}

{\small 
\begin{center}
\begin{tabular}{c}
\AXC{\ \ \ $\vdots$ \raisebox{1mm}{$\pi_1$}}
\noLine
\UI$\Gamma_1 \fCenter \GC (\varphi, \psi, \xi)$
\RL{\fns $g_R$}
\UI$\Gamma_1 \fCenter g (\varphi, \psi, \xi)$

\AXC{\ \ \ $\vdots$ \raisebox{1mm}{$\pi_2$}}
\noLine
\UI$\Gamma_2 \fCenter \varphi$
\AXC{\ \ \ $\vdots$ \raisebox{1mm}{$\pi_3$}}
\noLine
\UI$\psi \fCenter \Delta_3$
\AXC{\ \ \ $\vdots$ \raisebox{1mm}{$\pi_4$}}
\noLine
\UI$\xi \fCenter \Delta_4$

\LL{\fns $g_L$}
\TI$g (\varphi, \psi, \xi) \fCenter \GC (\Gamma_2, \Delta_3, \Delta_4)$
\RL{\fns Cut}
\BI$\Gamma_1 \fCenter \GC (\Gamma_2, \Delta_3, \Delta_4)$
\DP

 \\

\rule{0pt}{3.25ex}\rotatebox[origin=c]{-90}{$\rightsquigarrow$} \\

\AXC{\ \ \ $\vdots$ \raisebox{1mm}{$\pi_2$}}
\noLine
\UI$\Gamma_2 \fCenter \varphi$
\AXC{\ \ \ $\vdots$ \raisebox{1mm}{$\pi_1$}}
\noLine
\UI$\Gamma_1 \fCenter \GC (\varphi, \psi, \xi)$
\RL{\fns $(\GC, \GCF_1)$}
\UI$\varphi \fCenter \GCF_1 (\Gamma_1, \psi, \xi)$
\RL{\fns Cut}
\BI$\Gamma_2 \fCenter \GCF_1 (\Gamma_1, \psi, \xi)$
\RL{\fns $(\GCF_1, \GC)$}
\UI$\Gamma_1 \fCenter \GC (\Gamma_2, \psi, \xi)$
\RL{\fns $\GHF_2 \dashv \GC$}
\UI$\GHF_2 (\Gamma_2, \Gamma_1, \xi) \fCenter \psi$

\AXC{\ \ \ $\vdots$ \raisebox{1mm}{$\pi_3$}}
\noLine
\UI$\psi \fCenter \Delta_3$
\RL{\fns Cut}
\BI$\GHF_2 (\Gamma_2, \Gamma_1, \xi) \fCenter \Delta_3$
\LL{\fns $\GHF_2 \dashv \GC$}
\UI$\Gamma_1 \fCenter \GC (\Gamma_2, \Delta_3, \xi)$
\RL{\fns $\GHF_3 \dashv \GC$}
\UI$\GHF_3 (\Gamma_2, \Delta_3 \Gamma_1) \fCenter \xi$

\AXC{\ \ \ $\vdots$ \raisebox{1mm}{$\pi_4$}}
\noLine
\UI$\xi \fCenter \Delta_4$
\RL{\fns Cut}
\BI$\GHF_3 (\Gamma_2,\Delta_3, \Gamma_1) \fCenter \Delta_4$
\LL{\fns $\GHF_3 \dashv \GC$}
\UI$\Gamma_1 \fCenter \GC (\Gamma_2, \Delta_3, \Delta_4)$
\DP
 \\
\end{tabular}
\end{center}
 }

\begin{itemize}
\item[] The cut formulas of the cut inferences in the new proof are strict subformulas of the cut inference in the original proof, therefore the complexity of all the cut inferences in the new proof is lower than the complexity of the cut inference in the original proof. The proof transformation above relies on two ingredients: (i) the operational rules of $g$ exhibit the principal and the auxiliary formulas in display, (ii) the display rules apply to $\GC$ in each coordinate. These ingredients are available per each connective in the language of $\mathrm{D.LE^*}$ by the definition of operational rules and the fact that display property holds for $\mathrm{D.LE^*}$, so the proof transformation exemplified above immediately scales to arbitrary principal cut formulas. Therefore, condition C$_8$ is satisfied. 
\end{itemize}

\section{Gaggle theory: basic notions and terminology}
\label{sec:GagglesTheoryBasicNotionsAndNomenclature}

In \cite{Gore98Gaggles} (see also \cite{gore1998substructural}) a modified presentation of basic notions introduced in \cite{Dunn93} is provided, using a specific terminology and notational conventions. This terminology is subsequently adopted in a series of publications we collectively refer to as gaggle theory literature. In this paper, we adopt the terminology and (slightly modified presentations of) basic notions adopted in a series of publications we collectively refer to as order theory (or algebraic logic) literature (see, for instance, \cite{JonTar51,DaveyPriestley02,GalJipKowOno07}). In this appendix, we facilitate the translation between gaggle theory and order theory terminologies and we compare the basic notions. 

Below we provide slightly modified (but equivalent) definitions of basic notions in the gaggle theory literature as introduced in \cite{Gore98Gaggles}[Section 2]. Notice that the terminology and notational conventions are exactly like in \cite{Gore98Gaggles}[Section 2].

First we need the following
\begin{definition}[Isotone, antitone]
\label{def:IsotonicAntitonic}
An operation $\ell$ is called
\begin{itemize}
\item {\em isotonic in the j-th position} if for all $a, b \in A$, 
\begin{center}
$a \leq b \Rightarrow \ell\, (a_1, \ldots, a_{j-1}, a, a_{j+1}, \ldots, a_n) \leq \ell \, (a_1, \ldots, a_{j-1}, b, a_{j+1}, \ldots, a_n)$
\end{center}
\item{\em antitonic in the j-th position} if for all $a, b \in A$, 
\begin{center}
$a \leq b \Rightarrow \ell\, (a_1, \ldots, a_{j-1}, b, a_{j+1}, \ldots, a_n) \leq \ell \, (a_1, \ldots, a_{j-1}, a, a_{j+1}, \ldots, a_n)$.
\end{center}
\end{itemize}
\end{definition}

We now provide a modified, but equivalent, presentation of tonicity.\footnote{Notice that in the gaggle theory literature the tonicity of an operation $\ell$ is rather presented as a predicate $\textrm{tn}(\ell, j, \pm)$ where $\pm \in \{-, +\}$. In this paper we define $\textrm{tn}(\ell, j) = -$ iff $\textrm{tn}(\ell, j, -)$ is true (resp.~$\textrm{tn}(\ell, j) = +$ iff $\textrm{tn}(\ell, j, +)$ is true).}

\begin{definition}
\label{def:Tonicity}
The {\em tonicity} of an operation $\ell$ is a function assigning a value in the two-element set $\{+, -\}$ to each coordinate $j$ of the operation $\ell$, namely $\textrm{tn}(\ell, j) = +$ or $\textrm{tn}(\ell, j) = -$. The value $+$ (resp.~$-$) encodes the fact that $\ell$ is an {\em isotonic} (resp.~{\em antitonic}) operation in the $j$-th coordinate (see Definition \ref{def:IsotonicAntitonic}).
\end{definition}

Notice that, as explained in Section \ref{ssec:basic}, the {\em order type} $\epsilon$ of a connective $\ell$ is a function assigning a value in the two-element set $\{1, \partial\}$ to each coordinate $j$ of the connective $\ell$, namely $\epsilon_{\ell, j} = 1$ or $\epsilon_{\ell, j} = \partial$. The value $1$ (resp.~$\partial$) encodes the fact that the interpretation of the connective $\ell$ (namely the operation $\ell^{\mathbb{A}}$ where $\mathbb{A}$ is the relevant algebraic semantics) is an order-preserving (resp.~order-reversing) function in the $j$-th coordinate. 

We preliminarly introduce the following auxiliary definition: 

\begin{definition}[Normal operations]
\label{def:NormalOperations}
An {\em $f$-normal} operation $\ell$ has a tonicity and it preserves $\abot$ in the isotone coordinates and transforms $\aatop$ into $\abot$ in the antitone coordinates, namely
\begin{itemize}
\item if $\textrm{tn}(\ell, j) = +$, then $\ell(a_1\ldots, \abot_j, \ldots a_{n}) = \abot$,
\item if $\textrm{tn}(\ell, j) = -$, then $\ell(a_1\ldots, \aatop_j, \ldots a_{n}) = \abot$.
\end{itemize}

A {\em $g$-normal} operation $\ell$ has a tonicity and it preserves $\aatop$ in the isotone coordinates and transforms $\abot$ into $\aatop$ in the antitone coordinates, namely
\begin{itemize}
\item if $\textrm{tn}(\ell, j) = +$, then $\ell(a_1\ldots, \aatop_j, \ldots a_{n}) = \aatop$,
\item if $\textrm{tn}(\ell, j) = -$, then $\ell(a_1\ldots, \abot_j, \ldots a_{n}) = \aatop$.
\end{itemize}

A {\em normal} operation is either an $f$-normal or a $g$-normal operation.\footnote{Notice that graded modalities are natural examples of normal operations (see \cite{gradedgoble,Fin72,PanTzi22}).}
\end{definition}

We are now ready to provide the following:

\begin{definition}[Trace]
We postulate that $+- = -+ = -$ and $++ = -- = +$. Given an $n$-ary operation $\ell$, and the tonicity of $\ell$ per each coordinate (namely the vector of values $(\pm_1, \pm_2, \ldots, \pm_n)$ where $\pm \in \{+, -\}$), the {\em trace} of $\ell$ is defined as follows:
\begin{itemize}
\item $\textrm{tr}(\ell) = (-\pm_1, -\pm_2, \ldots, -\pm_n) \mapsto -$, if $\ell$ is an $f$-normal operation,
\item $\textrm{tr}(\ell) = (+\pm_1, +\pm_2, \ldots, +\pm_n) \mapsto +$, if $\ell$ is a $g$-normal operation.
\end{itemize}
\end{definition}

We preliminarily introduce the following auxiliary definitions: 

\begin{definition}
For any two operations $k,\ell$ of arity $n$ and a family of operators $OP$, we say that 
\begin{itemize}
\item $k$ is the {\em contrapositive} of $\ell$ in the $j$-th coordinate whenever 
\begin{center}
$\ell : (\pm_1, \ldots, \pm_j ,\ldots, \pm_n) \mapsto \pm_{n+1} \ \Rightarrow \ k : (\pm_1, \ldots, -\pm_{n+1}, \ldots, -\pm_n) \mapsto -\pm_j$;
\end{center}
\item $k$ and $\ell$ satisfy the {\em Abstract Law of Residuation} in their $j$-th coordinate if 
\begin{itemize}
\item $k$ and $\ell$ are contrapositive,
\item $\ell (a_1, \ldots, a_j ,\ldots, a_n) \leq b \Leftrightarrow a_j \leq k (a_1, \ldots, b ,\ldots, a_n)$ if $\ell$ is an $f$-normal operation,
\item $b \leq \ell (a_1, \ldots, a_j ,\ldots, a_n) \Leftrightarrow k (a_1, \ldots, b ,\ldots, a_n) 
\leq a_j$ if $\ell$ is a $g$-normal operation;
\end{itemize}
\item $k, \ell \in OP$ are {\em relatives} if they satisfy the Abstract Law of Residuation in some coordinate;
\item $OP$ is {\em founded} if there is a distinguished operation $\ell \in OP$ called the {\em head} and any other operation $k \in OP$ is a relative of $\ell$.
\end{itemize}
\end{definition}

We are now ready to define

\begin{definition}
\label{def:TonoidPartialGaggleFrame}
A structure $\mathcal{T} := (A, \leq, OP)$ is a {\em tonoid}\footnote{See \cite{Dun90} for the definition of {\em distributoid}.} if $(A, \leq)$ is a non-empty partially ordered set, and each operation in $OP$ has finite arity and it is a normal operation (therefore it has a trace). 

A {\em partial gaggle}\footnote{Notice that a {\em gaggle} is a partial gaggle with a distributive lattice reduct (see \cite{Dun90,Dunn93}).} is a tonoid where $OP$ is a founded family.

A structure $\mathcal{T} := (U, \sqsubseteq, \langle R_i \rangle_{i \in I})$ is a {\em frame} if $(U, \sqsubseteq)$ is a non-empty partially ordered set, and each relation $R_i$ of arity $n+1$ is the relation associated to the operation $\ell$ of arity $n$ in the corresponding tonoid (resp.~partial gaggle) $\mathcal{T}$.\footnote{Notice that the structures called frames in Definition \ref{def:TonoidPartialGaggleFrame} are, indeed, ordered frames, therefore the lattice fragment, if present, is distributive.}
\end{definition}

In \cite{Dunn93} it is proved that every partial gaggle can be represented as a frame (in the sense of Definition \ref{def:TonoidPartialGaggleFrame}), and in \cite{Gore98Gaggles} it is proved that every basic LE-logic (see Section \ref{ssec:basic}) can be captured by a cut-free basic display calculus. Two natural questions in the context of gaggle theory are the following. 
\begin{itemize}
\item[(i)] Consider a logic $\mathbf{L}$ in the language $\mathcal{L}$ and a display calculus $\mathbf{D.L}$ with logical language $\mathcal{L}$ and structural language $\mathcal{L}^\ast$, where the operations interpreting the connectives in $\mathcal{L}^\ast$ are fully-founded but the operations interpreting the connectives in $\mathcal{L}$ are not, and such that $\mathbf{D.L}$ derives all the theorems of $\mathbf{L}$. Is the logic of $\mathbf{D.L}$ a conservative extension of $\mathbf{L}$?\footnote{This question is explicitly stated as open problem in \cite{Gore98Gaggles}[Section 5.1]}
\item[(ii)] Consider a logic $\mathbf{L}$ in the language $\mathcal{L}$ that can be presented via a display calculus $\mathbf{D.L}$. Consider the logic $\mathbf{L'} = \mathbf{L} \cup \Sigma$, where $\Sigma$ is a set of axioms in the language $\mathcal{L}$. Can we provide a display calculus $\mathbf{D.L'}$ capturing $\mathbf{L'}$?
\end{itemize}

In \cite{greco2018unified} questions (i) and (ii) are answered in the positive.\footnote{The answer to question (i) is the statement of Theorem \ref{th:conservative extension} in Section \ref{ssec:basic}, where we reproduce the proof as well. Question (ii) is also addressed in \cite{CiRa14}. See \cite{greco2018unified}[Section 9] for a comparison between the characterisations of display calculi provided in \cite{greco2018unified} and \cite{CiRa14}. See \cite{ChnGrePalTzi21} for an overview of the literature on automatic rule generation in the context of structural proof theory.} 

We finally provide a translation table between order theory and gaggle theory terminology:
\begin{center}
{\small
\begin{tabular}{@{}rcl@{}}
order theory &  & gaggle theory \\
\hline
order-type & $\Leftrightarrow$ & tonicity \\
$\epsilon_{\ell, j} = 1$ & $\Leftrightarrow$ & $\textrm{tn}(\ell^{\mathbb{A}}, j) = +$ \\
$\epsilon_{\ell, j} = \partial$ & $\Leftrightarrow$ & $\textrm{tn}(\ell^{\mathbb{A}}, j) = -$ \\
$\ell^{\mathbb{A}}$ is order-pres.~in the $j$ coord. & $\Leftrightarrow$ & $\ell^{\mathbb{A}}$ is isotonic in the $j$ coord. \\
$\ell^{\mathbb{A}}$ is order-rev.~in the $j$-th coord. & $\Leftrightarrow$ & $\ell^{\mathbb{A}}$ is antitonic in the $j$ coord. \\
$\ell^{\mathbb{A}}$ is residuated in the $j$-th coord. & $\Leftrightarrow$ & $\ell^{\mathbb{A}}$ is contrapositive in the $j$ coord. \\
$\ell \in \mathcal{F}$ & $\Rightarrow$ & $\textrm{tr}(\ell^{\mathbb{A}}) = (-\pm_1, -\pm_2, \ldots, -\pm_n) \mapsto -$\\
$\ell \in \mathcal{G}$ & $\Rightarrow$ & $\textrm{tr}(\ell^{\mathbb{A}}) = (+\pm_1, +\pm_2, \ldots, +\pm_n) \mapsto +$ \\
$\ell$ is a normal connective & $\Leftrightarrow$ & $\ell^{\mathbb{A}}$ is normal operation \\
$\ell$ is a normal connective & $\Leftrightarrow$ & $\ell^{\mathbb{A}}$ has a trace \\
residuation & $\Leftrightarrow$ & Abstract Law of Residuation \\
$k^{\mathbb{A}}$ and $\ell^{\mathbb{A}}$ are residuated in the $j$-th coord. & $\Leftrightarrow$ & $k^{\mathbb{A}}$ and $\ell^{\mathbb{A}}$ are relative \\
residuated family & $\Leftrightarrow$ & founded family \\
head & $\Leftrightarrow$ & head \\
\hline
\end{tabular}
 }
\end{center}

\begin{remark}
In the order theory literature, the following finer distinctions are sometimes used: (i) if an operation is unary, `$\ell^{\mathbb{A}}$ has an adjoint in the $j$-th coordinate' might be used instead of `$\ell^{\mathbb{A}}$ is residuated in the $j$-th coordinate'; (ii) if an operation is order-reversing it the $j$-th coordinate, `$\ell^{\mathbb{A}}$ is in a Galois connection in the $j$-th coordinate' might be used instead of `$\ell^{\mathbb{A}}$ is residuated in the $j$-th coordinate'.

\end{remark}

While the characterization results given in \cite{Kracht,greco2018unified,CiRa14} set hard boundaries to the scope of proper display calculi, the {\em multi-type methodology} refines and generalises the theory allowing to capture logics which are not properly displayable in their single-type presentation. Examples of such logics include very well known and widely used logical frameworks such as inquisitive logic \cite{inquisitive}, DEL \cite{FriGreKurPalTzi16}, PDL \cite{PDL}, semi De Morgan logic and some of its extensions \cite{SDM}, bilattice logic \cite{greco2017bilattice}, non normal and conditional logics \cite{CHEN2022104756}, and logics of rough algebras \cite{10.1007/978-3-662-58771-3_14,GLMP18}. Moreover, the approach provides a natural environment for the design of new families of logics, such as those introduced in \cite{bilkova2018logic}. In this line of research, all the basic notions (i.e.~order type, order-preserving and order-reversing operations, operators and normal operators) are generalised to functions, namely operations the domain and codomain of which do not necessarily coincide (functions in this sense are called {\em heterogeneous} functions, resp.~operators, connectives, or modalities).


\end{document}